\documentclass[final]{siamltex}
\usepackage{amsmath,amsfonts,amssymb,mathtools}
\usepackage{latexsym}
\usepackage{graphicx,overpic}
\usepackage{caption}
\usepackage{pgfplots}
\pgfplotsset{compat=newest}
\usepackage{color}
\usepackage{booktabs}
\usepackage{array}
\newcolumntype{C}[1]{>{\centering\let\newline\\\arraybackslash\hspace{0pt}}m{#1}}
\usepackage[hidelinks]{hyperref}
\hypersetup{
    colorlinks=true,
    linkcolor={red!50!black},
    citecolor={blue!50!black},
    urlcolor={blue!80!black}
}

\usepackage{tikz}

\newcommand{\R}{\mathbb R}
\newcommand{\bA}{\mathbf A}
\newcommand{\bB}{\mathbf B}
\newcommand{\bC}{\mathbf C}

\newcommand{\bH}{\mathbf H}
\newcommand{\bI}{\mathbf I}

\newcommand{\bP}{\mathbf P}

\newcommand{\bM}{\mathbf M}

\newcommand{\bS}{\mathbf S}

\newcommand{\bV}{\mathbf V}

\newcommand{\bn}{\mathbf n}
\newcommand{\be}{\mathbf e}

\newcommand{\bu}{\mathbf u}
\newcommand{\bU}{\mathbf U}
\newcommand{\bv}{\mathbf v}

\newcommand{\bw}{\mathbf w}

\newcommand{\bz}{\mathbf z}
\newcommand{\bbf}{\mathbf f}

\newcommand{\T}{\mathcal T}

\newcommand{\tr}{{\rm tr}}
\newcommand{\Div}{\mathop{\rm div}}
\newcommand{\divG}{{\mathop{\,\rm div}}_{\Gamma}}
\newcommand{\divGh}{{\mathop{\,\rm div}}_{\Gamma_h}}
\newcommand{\gradG}{\nabla_{\Gamma}}

\newcommand{\gradGh}{\nabla_{\Gamma_h}}

\newcommand{\OGamma}{\Omega^\Gamma_h}

\newcommand{\Gammalin}{\Gamma^{\rm lin}}

\newcommand{\cT}{\mathcal T}

\newcommand{\vect}[1]{\boldsymbol{\mathbf{#1}}}

\newcommand{\Plin}{\mathbf{P}_{\rm lin}}
\newcommand*\diff{\mathop{}\!\mathrm{d}}

\usepackage{subcaption}
\usepackage{graphicx}
\graphicspath{{img/}} % includegraphics path
\newcommand{\includegraphicsw}[2][1.]{\includegraphics[width=#1\linewidth]{#2}}
\DeclareMathOperator{\Curl}{curl}
\DeclareMathOperator{\vCurl}{\vect{curl}}
\DeclareMathOperator{\Tr}{tr}

\newcommand{\LTwoSpace}[1][\Gamma]{{L^2\left({#1}\right)}}
\newcommand{\LInfSpace}[1][\Gamma]{{L^\infty\left({#1}\right)}}

\newcommand{\KinEn}{\mathcal{E}}

\numberwithin{equation}{section}
\begin{document}

\title{Error analysis of higher order trace finite element methods for the surface Stokes equations}
\author{Thomas Jankuhn\thanks{Institut f\"ur Geometrie und Praktische  Mathematik, RWTH-Aachen
University, D-52056 Aachen, Germany (jankuhn@igpm.rwth-aachen.de)}
\and Maxim A. Olshanskii\thanks{Department of Mathematics, University of Houston, Houston, Texas 77204 (molshan@math.uh.edu)}
\and Arnold Reusken\thanks{Institut f\"ur Geometrie und Praktische  Mathematik, RWTH-Aachen
University, D-52056 Aachen, Germany (reusken@igpm.rwth-aachen.de).}
\and Alexander Zhiliakov\thanks{Department of Mathematics, University of Houston, Houston, Texas 77204 (alex@math.uh.edu)}
}
\maketitle

\begin{abstract} The paper studies a higher order unfitted finite element method for the Stokes system posed on a surface in $\mathbb{R}^3$. The method employs parametric $\mathbf{P}_{k}$--$P_{k-1}$ finite element pairs on tetrahedral bulk mesh to discretize the Stokes system on embedded surface. Stability and optimal order convergence results are proved. The proofs include a complete quantification of geometric errors stemming from approximate parametric representation of the surface.
Numerical experiments include formal convergence studies and an example of the Kelvin--Helmholtz instability problem on the unit sphere.
\end{abstract}
\section{Introduction}
Fluid equations posed on manifolds arise in continuum based  models of thin material layers with lateral viscosity such as lipid monolayers and plasma membranes~\cite{GurtinMurdoch75,arroyo2009,rangamani2013interaction,torres2019modelling}. Beyond biological sciences, fluid equations  on  surfaces appear in the literature on modeling of foams,  emulsions and liquid crystals; see, e.g., \cite{scriven1960dynamics,slattery2007interfacial,fuller2012complex,brenner2013interfacial,rahimi2013curved,nitschke2019hydrodynamic}.
Despite the apparent practical and  mathematical relevance, such systems have received little attention  from the scientific computing community until the very recent series of publications \cite{nitschke2012finite,Jankuhn1,reusken2018stream, reuther2018solving,fries2018higher,olshanskii2018finite,olshanskii2019penalty,nitschke2019hydrodynamic,gross2019meshfree,bonito2019divergence,jankuhn2019higher,olshanskii2019inf,brandner2019finite,lederer2019divergence} that evidences a strongly growing interest  in the development and analysis of numerical methods for  fluid equations posed on surfaces.

Discretization of fluid systems on manifolds brings up several difficulties in addition to those well-known for equations posed in Euclidian domains. First, one has to approximate covariant derivatives.
Another difficulty stems from the need to recover a \emph{tangential} velocity field on a surface $\Gamma$.  It is not straightforward to build a finite element method (FEM), which is conformal with respect to this tangentiality condition.
Two natural ways to enforce the condition in the numerical setting are either to use Lagrange multipliers or add a penalty term to the weak variational formulation. Next, one has to deal with geometric errors originating from approximation of $\Gamma$ by a ``discrete''  surface $\Gamma_h$ or, more general, from inexact integration  over $\Gamma$.

Among recent publications, Ref.~\cite{reuther2018solving,fries2018higher} applied surface  FEMs to discretize the incompressible surface Navier--Stokes equations in primitive variables on stationary triangulated  manifolds.
In \cite{reuther2018solving}, the authors considered $\vect P_1$--$P_1$ finite elements without pressure stabilization
and with a  penalty technique to force the flow field to be approximately tangential to the surface. In \cite{fries2018higher}, instead, surface Taylor--Hood elements are used and combined with  a Lagrange multiplier method to satisfy the tangentiality constraint.  Divergence-free DG and H(div)-conforming finite element methods for the surface Stokes problem were recently introduced in \cite{lederer2019divergence,bonito2019divergence}. These methods enforce the tangentiality condition strongly.  In~\cite{gross2019meshfree} the authors suggest meshfree methods for hydrodynamic equations on steady curved surfaces. In  \cite{Sahu2020,Sanchez2019} special surface parametrizations are used, such that penalty and Lagrange multiplier  techniques for treating the tangential constraint can be avoided. Finally, yet another approach was taken in  \cite{nitschke2012finite,reusken2018stream},  where the governing equations were written in vorticity--stream function variables and surface finite element techniques available for scalar equations are applied. None of these references address the numerical analysis of  the discretization method.

First stability and error analyses of finite element formulations for the surface Stokes problem are presented only in very recent papers \cite{bonito2019divergence,olshanskii2018finite,olshanskii2019inf,brandner2019finite}.
The authors of \cite{bonito2019divergence} present an analysis of the lowest-order Brezzi-Douglas-Marini H(div)-conforming finite element. %They introduced  an interior penalty method to weakly maintain $H^1$ conformity of the velocity field. The resulting error bounds show second order $L^2$-convergence for velocity.
The surface Stokes problem  is discretized using unfitted stabilized $\mathbf{P}_1$--$P_1$ elements in \cite{olshanskii2018finite}, and
 the trace FEM with $\mathbf{P}_2$--$P_1$ bulk elements has been considered in \cite{olshanskii2019inf}. Both papers \cite{olshanskii2018finite,olshanskii2019inf} give a full convergence analysis, but assume exact numerical integration over the  surface.  In \cite{brandner2019finite} a convergence analysis of a surface finite element based on the
vorticity--stream function variables is presented.
In none of these papers on an unfitted FEM for  surface Stokes-type systems error bounds  including geometric consistency estimates are derived.

We consider a mixed trace FEM for the surface Stokes in pressure-velocity variables on a given smooth surface $\Gamma$ without boundary. In the trace FEM, polynomial functions defined on an ambient (bulk) mesh are used to set up trial and test spaces \cite{OlshReusken08,olshanskii2016trace}. For these bulk finite element spaces we shall consider the generalized Taylor--Hood elements ($\mathbf{P}_{k}$--$P_{k-1}$, $k\ge2$, elements on tetrahedra), which is known to be inf-sup stable in the bulk. To ensure that the  geometric error is consistent with the polynomial interpolation error, we employ a parametric version~\cite{lehrenfeld2016high,grande2017higher} of the trace finite elements. A penalty method is used to (approximately) satisfy the tangentiality constraint.
To approximate the tangential gradient and handle covariant derivatives, the method exploits the embedding of  $\Gamma$ in $\R^3$   and makes use of tangential differential calculus.
This allows us to avoid the use of intrinsic variables on a surface and makes implementation of the numerical method relatively straightforward.

The paper presents a stability  and convergence analysis, which accounts for both interpolation \emph{and geometric errors}. The analysis is not straightforward, since the uniform (with the respect of the surface position in the background mesh) inf-sup stability of trace spaces does not follow in any direct way from the stability of the bulk mixed elements. By quantifying geometric errors and extending results from~\cite{olshanskii2019inf}, we prove such an inf-sup stability condition for $\mathbf{P}_{k}$--$P_{k-1}$ elements, for \emph{arbitry} $k \geq 2$.  With the help of the stability result and  geometric consistency estimates  derived for a vector-Laplace problem in  \cite{jankuhn2019trace} we further derive FE error estimates  in a surface energy norm. The  error bound that we derive is optimal with respect to $h$ and uniform with respect to the position of the surface approximation $\Gamma_h$ in a background mesh.

Summarizing, the main contributions of this paper are: 1. we extend the analysis from \cite{olshanskii2019inf} (for $k=2$) to \emph{higher order Taylor--Hood elements} ($k \geq 2$); 2. we prove stability and optimal order error estimates \emph{including the effect of geometric errors}.  Results of extensive numerical experiments with the parametric unfitted finite element that we analyze in this paper are given in~\cite{jankuhn2019higher}. These results confirm the optimal convergence orders of the trace generalized Taylor--Hood elements.
We give a further numerical assessment of the entire approach in terms of eigenvalue computations and an application with a surface Navier--Stokes equations with a high Reynolds number.

The remainder of the paper is organized as follows. In section~\ref{sectioncont} we recall some basics of tangential differential calculus and formulate the surface Stokes system, our problem of interest.
In section~\ref{sectparametric} parametric trace finite element spaces are explained together with there properties necessary for further analysis. The finite element discretization of the surface Stokes system  is given in section~\ref{sectFEmethods}. Its well-posedness is analyzed in section~\ref{sectwellposed}. In the subsection~\ref{sectstab} we prove one of our key results concerning inf-sup stability of the velocity--pressure FE spaces. We proceed with the error analysis in section~\ref{sectAnalysis}. It includes a complete quantification of the geometric error, which makes it rather technical. Section~\ref{sectNumer} contains resuts of
numerical experiments illustrating certain properties of the method.

\section{Surface Stokes problem} \label{sectioncont}
Consider  a   smooth  hypersurface $\Gamma\subset \mathbb{R}^3$, which is connected, closed and compact. %embedded in a domain
We further assume the implicit representation of $\Gamma$ as the zero level of a smooth level set function $\phi \colon U_\delta \to \mathbb{R}$, i.e.
\[
\Gamma = \{ x \in \Omega \mid \phi(x) =0 \}\quad\text{and}\quad |\nabla \phi(x)|\ge c_0>0
\]
for all  $x$ in $U_\delta$, a tubular $\delta$-neighborhood of $\Gamma$.
 We assume $\delta>0$ to be sufficiently small such that  for any $x \in U_{\delta}$ the following quantities are well defined: $d(x)$ the smooth signed distance function to $\Gamma$, negative in the interior of $\Gamma$; $\bn(x) = \nabla d(x)$, the extension of the outward normal vector on $\Gamma$;  $\bH(x) = \nabla^2d(x)$, the Weingarten map;  $\bP(x):= \bI - \bn(x)\bn(x)^T$, the orthogonal projection onto the tangential plane; and  $p(x) = x - d(x)\bn(x)$, the closest point mapping from $U_\delta$ on $\Gamma$.

 We associate any scalar or vector function $g$ on $\Gamma$  with its normal extension in $U_{\delta}$ defined as $g^e(x) := g(p(x))$, $x \in U_{\delta}$.
The Sobolev norms of the normal extension $g^e$ on any $\epsilon$-neighborhood,   $\mathcal{O}_\epsilon= \left\lbrace x \in \mathbb{R}^3 \mid \vert d(x) \vert < \epsilon\right\rbrace$, $0<\epsilon\le\delta$   are estimated by the corresponding norms on $\Gamma$~\cite{reusken2015analysis} as
\begin{equation} \label{lemmasobolevnormsneighborhood}
\Vert D^\mu g^e \Vert_{L^2(\mathcal{O}_\epsilon)} \lesssim \epsilon^{\frac{1}{2}} \Vert g \Vert_{H^m(\Gamma)} \qquad \text{for all} ~ g \in H^m(\Gamma), \, \vert \mu \vert \leq m.
\end{equation}
We shall skip the superscript and use the same notation for a function and its extension, if no confusion arises.
For a scalar field $\psi$, a vector field $\bu$ on $\Gamma$ and tensor field $\bA: \Gamma \to \Bbb{R}^{n\times n}$, one then can define the surface gradient,  divergence, covariant gradient, the surface rate-of-strain tensor (see \cite{GurtinMurdoch75}):
\[
\gradG \psi=\bP\nabla \psi,\quad \divG\bu=\mbox{tr}(\bP\nabla\bu), \quad
\gradG\bu=\bP(\nabla \bu)\bP,\quad E(\bu):= \frac{1}{2} \left( \gradG \bu + \gradG^T \bu \right),
\]
and the surface  divergence operator $\divG \bA   := \left( \divG (\be_1^T \bA),\,
               \dots, \,
               \divG (\be_n^T \bA)\right)^T$.

The \emph{surface Stokes problem} reads: For a given force vector $\bbf \in L^2(\Gamma)^3$, with $\bbf \cdot \bn =0$, and a source term $g \in L^2(\Gamma)$, with $\int_\Gamma g\, ds =0$, solve
\begin{equation} \label{eqstrong} \begin{split}
- \bP \divG (E(\bu)) + \bu + \gradG p &= \bbf \qquad \text{on } \Gamma,  \\
\divG \bu &= g \qquad \text{on } \Gamma,
\end{split}
\end{equation}
for a tangential velocity field $\bu \colon \Gamma \to \R^3$, $\bu \cdot \bn = 0$, and surface pressure $p \colon \Gamma \to \R$ with $\int_\Gamma p\, ds =0$.
We added the zero order term  to avoid technical details related to the kernel of the strain tensor $E$ (the so-called Killing vector fields).

For the weak formulation of \eqref{eqstrong}, we need the surface Sobolev space
\begin{equation} \label{eqdefH1}
\begin{gathered}
\bV:= H^1(\Gamma)^3, \quad \text{with} ~ \Vert \bu \Vert_{H^1(\Gamma)}^2 := \int_{\Gamma} \Vert \bu(s) \Vert_2^2 + \Vert \nabla \bu(s) \Vert_2^2 \, ds,
\end{gathered}
\end{equation}
and the  subspace of \emph{tangential} vector fields,
$ %\begin{equation*}
\bV_T := \left\lbrace \bu \in \bV \mid \bu \cdot \bn =0 \right\rbrace.
$ %\end{equation*}
For the orthogonal decomposition of $\bv\in\bV$  into a tangential and a normal part, we use the notation:
%\begin{equation*}
$
\bv =  \bv_T + v_N\bn,
$
with $\bv_T=\bP \bv$ and  $v_N=\bv \cdot \bn$.
%\end{equation*}
For $\textbf{u}, \textbf{v} \in \bV$ and $q \in L^2(\Gamma$) consider the bilinear forms
\[
a(\textbf{u}, \textbf{v}) := \int_\Gamma E(\bu) : E(\bv) \, ds + \int_\Gamma \textbf{u} \cdot \textbf{v} \, ds,\quad
b(\textbf{u}, q) := - \int_\Gamma q \divG \bu \, ds,
\]
%Note that in the definition of $b(\bu,p)$ only the \emph{tangential} component of $\bu$ is used, i.e., $b(\bu,p)=b(\bu_T,p)$ for all $\bu \in \bV$, $p\in L^2(\Gamma)$. This property motivates the notation $b(\cdot,\cdot)$ instead of $b(\cdot,\cdot)$.
and the following weak formulation of \eqref{eqstrong}: Find $(\bu, p)  \in \bV_T \times L^2_0(\Gamma)$ such that
\begin{equation} \label{contform} \tag{C} \begin{split}
a(\bu, \bv) + b(\bv, p) &= (\bbf, \bv)_{L^2(\Gamma)} ~~~ \text{for all}~ \bv \in \bV_T, \\
b(\bu,q) &= (-g,q)_{L^2(\Gamma)} ~~~ \text{for all}~ q \in L^2(\Gamma).
\end{split}
\end{equation}
The bilinear form $a(\cdot,\cdot)$ is continuous on $\bV$, and hence on $\bV_T$. The ellipticity of $a(\cdot,\cdot)$ on $\bV_T$ follows from the following surface Korn inequality that holds if $\Gamma$ is $C^2$ smooth (cf., (4.8) in \cite{Jankuhn1}):
 There exists a constant $c_K > 0$ such that
\begin{equation} \label{Korn}
\Vert \bu \Vert_{L^2(\Gamma)} + \Vert E(\bu) \Vert_{L^2(\Gamma)} \geq c_K \Vert \bu \Vert_{H^1(\Gamma)} \qquad \text{for all } \bu \in \bV_T.
\end{equation}
The bilinear form $b(\cdot,\cdot)$ is continuous on $\bV_T \times L_0^2(\Gamma)$ and satisfies the following inf-sup condition (Lemma 4.2 in \cite{Jankuhn1}):
There exists a constant $c_0>0$ such that estimate
\begin{equation}\label{LBB}
\inf_{p \in L^2_0(\Gamma)} \sup_{\bv \in \bV_T} \frac{b(\bv, p)}{\Vert \bv\Vert_{H^1(\Gamma)} \Vert p \Vert_{L^2(\Gamma)}} \geq c_0,
\end{equation}
holds.
Hence, the weak formulation \eqref{contform} is \emph{a well-posed problem}. Its unique solution is denoted by $(\bu^*,p^*)$.

\section{Parametric finite element spaces for high order surface approximation} \label{sectparametric}
Let $\{ \mathcal{T}_h \}_{h>0}$ be a family of shape regular tetrahedral triangulations of a polygonal domain $\Omega \subset \mathbb{R}^3$ that contains the surface $\Gamma$.
By $V_h^k$ we denote the standard finite element space of continuous piecewise polynomials of degree $k$. Denote by $I^k$ the nodal interpolation operator from $C(\overline{\Omega})$ to $V_h^k$.
In the original trace FEM introduced in~\cite{OlshReusken08} and analyzed for higher order elements in \cite{reusken2015analysis}, one uses the traces of functions from $V_h^k$ on $\Gamma_h \approx \Gamma$ to define trial and test FE spaces.
For a higher order finite element method, geometrical consistency order dictates that $\Gamma_h$ should be a sufficiently accurate approximation of  $\Gamma$. The latter poses the challenge of efficient numerical integration over the surface $\Gamma_h$, which is often defined implicitly, e.g.  as the zero level set of a higher order polynomial. We avoid this difficulty by using the \emph{parametric} trace FE approach as in \cite{grande2017higher, jankuhn2019higher}, which we outline below. %Below we outline the  parametric mapping and the corresponding FE spaces used in this method and summarize certain well-known properties.

Consider a FE level set function  $\phi_h \in V_h^{k}$ approximating  $\phi$ in the following sense:
\begin{equation} \label{phibound}
\max_{T \in \mathcal{T}_h} \vert \phi_h - \phi \vert_{W^{l,\infty}(T\cap U_\delta)} \leq c h^{k+1-l}, \quad 0\leq l \leq k+1,
\end{equation}
is satisfied.
Here, $\vert \cdot \vert_{W^{l,\infty}(T\cap U_\delta)}$ denotes the usual semi-norm on the Sobolev space $W^{l,\infty}(T\cap U_\delta)$ and the constant $c$ depends on $\phi$ but is independent of $h$. The zero level set of  $\phi_h$ \emph{implicitly} characterizes an approximation of the interface, i.e. for $k \geq 2$ no parametrization of this set is available  for  integration purposes. An \emph{easy to compute} piecewise-planar approximation of $\Gamma$ is provided by $\hat{\phi}_h = I^1\phi_h$:
\begin{equation*}
\Gamma^{\text{lin}} := \{ x \in \Omega \mid \hat{\phi}_h (x) = 0\}.
\end{equation*}
 Using $\Gamma^{\text{lin}}$ alone, however, limits the accuracy to second order. Hence one constructs a transformation of the bulk  mesh in $\Omega_h^\Gamma={\rm int}(\cup_{T\in\T^\Gamma_h}\overline{T})$, $\T^\Gamma_h=\{T\in \T_h\,|\,T\cap \Gamma^{\rm lin}\neq\emptyset\}$, with the help of an explicit mapping $\Theta_h$ parameterized by a finite element function, i.e., $\Theta_h\in  \big({V_h^{k}}_{|\OGamma}\big)^3$. The mapping $\Theta_h$ is such that $\Gamma^{\text{lin}}$ is mapped approximately to $\Gamma$; see \cite{grande2017higher,lehrenfeld2016high} for how $\Theta_h$ is constructed.
Hence, the parametric mapping $\Theta_h$ indeed yields a higher order, yet computable, surface approximation
\begin{equation*}
\Gamma_h := \Theta_h(\Gamma^{\text{lin}}) = \left\lbrace x \mid \hat{\phi}_h(\Theta_h^{-1}(x)) = 0 \right\rbrace.
\end{equation*}
In \cite{lehrenfeld2017analysis} it is shown that under reasonable smoothness assumptions the estimate
\begin{equation} \label{distres}
  {\rm dist}(\Gamma_{h}, \Gamma) \lesssim h^{k+1}
\end{equation}
holds. Here and further in the paper we write $A \lesssim B$ to state that there exists a constant $c>0$, which is independent of the mesh parameter $h$ and the position of $\Gamma$ in the background mesh, such that the inequality $A \leq c B$ holds.
We denote the transformed cut mesh domain by $\Omega^{\Gamma}_\Theta := \Theta_h(\Omega^\Gamma_h)$ and apply to $V_h^{k}$ the transformation $\Theta_h$ resulting in the \emph{parametric spaces} (defined on $\Omega^{\Gamma}_\Theta$)
\begin{equation*}
V_{h,\Theta}^{k} := \left\lbrace v_h \circ (\Theta_h)^{-1} \mid v_h \in {V_h^{k}}_{|\Omega^{\Gamma}_h} \right\rbrace, \quad
\bV_{h,\Theta}^{k} := (V_{h,\Theta}^{k})^3.
\end{equation*}

We recall some well-known approximation results from the literature \cite{grande2017higher}.  The parametric interpolation $I_{\Theta}^{k} \colon C(\Omega_{\Theta}^{\Gamma}) \to V_{h,\Theta}^{k}$ is defined by $(I_{\Theta}^{k} v)\circ \Theta_h = I^{k}(v\circ \Theta_h)$, with $I^k$ the standard nodal interpolation in $V_h^k$. We have the following optimal interpolation error bound for $0\leq l \leq {k}+1$:
\begin{equation} \label{eqinterpolationerror}
\Vert v - I_{\Theta}^{k} v \Vert_{H^l(\Theta_h(T))} \lesssim h^{k+1-l} \Vert v  \Vert_{H^{k+1}(\Theta_h(T))} ~~~\text{for all} ~ v \in H^{k+1}(\Theta_h(T)), T \in \mathcal{T}_h.
\end{equation}
For $\Gamma_T := \Gamma_h \cap \Theta_h(T)$,  we also need the following trace inequality~\cite{Hansbo02}:
\begin{equation} \label{eqtraceestimate}
\Vert v \Vert_{L^2(\Gamma_T)}^2 \lesssim h^{-1} \Vert  v \Vert_{L^2(\Theta_h(T))}^2 + h \Vert  \nabla v \Vert_{L^2(\Theta_h(T))}^2 ~~~ \text{for} ~ v \in H^1(\Theta_h(T)),
\end{equation}
The inequality remains true with $\Gamma^{\rm lin}$ and $T$ in place of $\Gamma_h$ and $\Theta_h(T)$.

The following approximation result for trace spaces is proved by standard arguments (cf. \cite{grande2017higher}), based on \eqref{eqinterpolationerror}, \eqref{eqtraceestimate} and \eqref{lemmasobolevnormsneighborhood} with $\epsilon=h$.
\begin{lemma} \label{lemmascalarapproximationerror}
For the space $V_{h,\Theta}^{k}$ we have the approximation property
\begin{multline*} %\begin{split}
\min_{v_h \in V_{h,\Theta}^{k}} \left( \Vert v^e - v_h \Vert_{L^2(\Gamma_h)} + h \Vert \nabla( v^e - v_h) \Vert_{L^2(\Gamma_h)} \right)
\leq \Vert v^e - I_{\Theta}^{k} v^e \Vert_{L^2(\Gamma_h)}\\ + h \Vert \nabla( v^e - I_{\Theta}^{k} v^e) \Vert_{L^2(\Gamma_h)} \lesssim h^{k+1} \Vert v \Vert_{H^{k+1}(\Gamma)}\quad\text{for all}~v \in H^{k+1}(\Gamma).
%\end{split}
\end{multline*}

\end{lemma}

The next lemma, taken from \cite{grande2017higher}, gives an approximation error  for the normal approximation $\bn_h$, which is easy to compute and used in our FE formulation below.
\begin{lemma} \label{lemmanormals}
For $T \in \mathcal{T}^\Gamma_h$ and any $x \in T $ define
\begin{equation*}
\bn_{\rm lin}(T) := \frac{\nabla \hat{\phi}_{h|T}}{\Vert \nabla \hat{\phi}_{h|T}\Vert_2}, \quad \bn_h (\Theta(x)) := \frac{D\Theta_h(x)^{-T} \bn_{\rm lin}(T)}{\Vert D\Theta_h(x)^{-T} \bn_{\rm lin}(T) \Vert_2}.
\end{equation*}
Restricted to surface approximations the vector fields $\bn_{\rm lin}$ and $\bn_h$ are normals on $\Gamma^{\rm lin}$ and $\Gamma_h$, respectively. Moreover,
$\Vert \bn_h - \bn \Vert_{L^\infty(\Omega_\Theta^\Gamma)} \lesssim h^{k}$ holds.
\end{lemma}

We also define the lifting $u^l$ of a function $u$ defined   on $\Gamma_h$ by
$
u^l(p(x)) = u(x) \text{ for } x \in \Gamma_h,$  {and}
$u^l(x) = u^l(p(x)) \text{ for } x \in U_\delta.
$
The following equivalences are well known (see \cite{Dziuk88,jankuhn2019higher}) for $w \in H^1(\Gamma_h)$ and $\bv \in H^1(\Gamma_h)^3$ and we shall frequently use these
\begin{align*}
\Vert w \Vert_{L^2(\Gamma_h)} &\simeq \Vert w^l \Vert_{L^2(\Gamma)}, \qquad \Vert \nabla_{\Gamma_h} w \Vert_{L^2(\Gamma_h)} \simeq \Vert \nabla_\Gamma w^l \Vert_{L^2(\Gamma)}, \\
\Vert \bv \Vert_{L^2(\Gamma_h)} &\simeq \Vert \bv^l \Vert_{L^2(\Gamma)}, \qquad \Vert \nabla \bv^l \bP_h \Vert_{L^2(\Gamma_h)} \simeq \Vert \nabla \bv^l \bP \Vert_{L^2(\Gamma)}. %\label{eqnormequivalenceH1}
\end{align*}
A norm on $H^1(\Gamma_h)^3$ is defined using the component-wise lifting by
\begin{equation*}
\Vert \bu \Vert_{H^1(\Gamma_h)}^2 := \int_{\Gamma_h} \Vert \bu(s) \Vert_2^2 + \Vert \nabla \bu^l(s) \bP_h(s) \Vert_2^2 \, ds,
\end{equation*}
with $\bP_h = \bI - \bn_h \bn_h^T$. Finally, we need the  following  spaces
\begin{equation*} \begin{split}
V_{reg,h} := \left\lbrace v \in H^1(\Omega_{\Theta}^{\Gamma}) \mid \text{tr} |_{\Gamma_h} v \in H^1(\Gamma_h) \right\rbrace \supset V_{h,\Theta}^{k}, \quad
\bV_{reg,h} := \big(V_{reg,h}\big)^3
\end{split}
\end{equation*}
and the ``discrete'' covariant gradient for $\bu\in\bV_{reg,h}$,
$
 \gradGh \bu := \bP_h \nabla \bu \bP_h.
$

\section{Higher order  trace finite element methods} \label{sectFEmethods}
 Based on the parametric finite element spaces $\bV_{h,\Theta}^{k}$ and $V_{h,\Theta}^{k}$ we consider for $k \geq 2$ the $\vect P_k$--$P_{k-1}$ pair of \emph{parametric trace Taylor--Hood elements}:
\begin{equation*}
\bU_h := \bV_{h,\Theta}^{k}, \qquad Q_h := V_{h,\Theta}^{k-1} \cap L^2_0(\Gamma_h).
\end{equation*}
Note that  the polynomial degrees, $k$ and $k-1$, for the velocity and pressure approximation are different, but both spaces $\bU_h$ and $Q_h$ use the same parametric mapping based on polynomials of degree $k$. Since the pressure approximation uses $H^1$ finite element functions we can use the integration by parts
 $
b(\bu_T,p) = \int_{\Gamma}  \bu \cdot \gradG p \, ds,
$ and replace $\Gamma$ by $\Gamma_h$ in the definition of the FE bilinear form.
Furthermore, recalling the identity $E(\bu_T)= E(\bu)-u_N \bH$ for $\bu=\bu_T+u_N\bn$ on $\Gamma$, we define the discrete rate-of-strain tensor by
\[
  E_{h}(\bu):=\frac12 \big(\gradGh \bu + \gradGh^T \bu\big) - (\bu \cdot \bn_h) \bH_h.
\]
%with $\gradGh \bu:= \bP_h \nabla \bu^l \bP_h$.
We introduce   the following FE variants of the bilinear forms $a(\cdot,\cdot)$,  $b(\bP\cdot,\cdot)$ and the penalty form $k(\cdot,\cdot)$:
\begin{align*}
 a_{h}(\bu,\bv) &:= \int_{\Gamma_h} E_{h}(\bu):E_{h}(\bv)\, ds_h + \int_{\Gamma_h} \bP_h \bu \cdot \bP_h \bv \, ds_h,\\
  b_h(\bu,q)& := \int_{\Gamma_h} \bu \cdot \gradGh q \, ds_h,\quad
 k_h(\bu,\bv):= \eta \int_{\Gamma_h} (\bu \cdot \tilde{\bn}_h) (\bv \cdot \tilde{\bn}_h)  \, ds_h.
\end{align*}
The bilinear form $k_h(\cdot,\cdot)$ is used to enforce (approximately) the condition $\bu \cdot \bn=0$.
The normal vector used in this bilinear form, and the curvature tensor $\bH_h$ are approximations of the exact normal and the exact Weingarten mapping, respectively.
The reason that we introduce yet another normal approximation $\tilde{\bn}_h$ is the following. From an error  analysis of the vector-Laplace problem in \cite{hansbo2016analysis,jankuhn2019higher} it follows that for obtaining optimal order estimates  the normal approximation $\tilde{\bn}_h$ used in the penalty term has to be more accurate than the normal approximation $\bn_h$. We assume
\begin{align*}
\Vert \bn - \tilde{\bn}_h \Vert_{L^{\infty}(\Gamma_h)} \lesssim h^{k+1}~~\text{and}~~
\Vert \bH - \bH_h \Vert_{L^{\infty}(\Gamma_h)} \lesssim h^{k-1}. %\label{eqapproxH}
\end{align*}
Since the trace FEM is a geometrically unfitted FE method, we need a stabilization that eliminates instabilities caused by small cuts. For this we use the so-called ``normal derivative volume stabilization'', known from the literature \cite{burmanembedded,grande2017higher}:
  \[
  s_h(\bu,\bv)  := \rho_u\int_{\Omega_{\Theta}^{\Gamma}} (\nabla \bu \bn_h) \cdot (\nabla \bv \bn_h)  \, dx, \quad
  \tilde{s}_h(p,q) := \rho_p  \int_{\Omega_{\Theta}^{\Gamma}} (\bn_h\cdot \nabla p) (\bn_h \cdot \nabla q) \, dx.
\]
The choice of the stabilization parameters $\rho_u,\, \rho_p$ will be discussed below; see \eqref{choicerho}.
\\[1ex]
For a suitable (sufficiently accurate) extension of the data $\bbf$ and $g$ to $\Gamma_h$, denoted by $\bbf_h$ and $g_h$, the finite element method reads:
Find $(\bu_h, p_h) \in \bU_h \times Q_h$ such that
\begin{equation} \tag{FEM} \label{discreteform1}
\begin{aligned}
A_h(\bu_h,\bv_h) + b_h(\bv_h,p_h) & =(\mathbf{f}_h,\bv_h)_{L^2(\Gamma_h)} &\quad &\text{for all } \bv_h \in \bU_h \\
b_h(\bu_h,q_h) - \tilde{s}_h(p_h,q_h) & = (-g_h,q_h)_{L^2(\Gamma_h)} &\quad &\text{for all }q_h \in Q_h,
\end{aligned}
\end{equation}
where $A_h(\bu,\bv) := a_{h}(\bu,\bv) + s_h(\bu,\bv)+k_h(\bu,\bv).$

In the error analysis below we use the following natural norms
\begin{equation} \label{defnorms} \begin{split}
\Vert \bu \Vert_{A}^2 := A_h(\bu, \bu), \qquad \Vert p \Vert_{M}^2 := \Vert p \Vert_{L^2(\Gamma_h)}^2 + \rho_p \Vert \bn_h \cdot \nabla p \Vert_{L^2(\Omega_{\Theta}^{\Gamma})}^2 .
\end{split}
\end{equation}
We address the choice of the stabilization parameters  $\rho_u$,
$\rho_p$ and the penalty parameter $\eta$. The analysis of optimal order error
bounds for vector-Laplace problem in \cite{jankuhn2019trace}  is restricted
to $\rho_u \simeq h^{-1}$, $\eta\simeq h^{-2}$.
Moreover, experiments discussed in \cite{jankuhn2019higher}
indicate that the choice $\rho_u \simeq h$ does not allow optimal
order error bounds. The stability analysis of trace $\mathbf{P}_2$--$P_1$ Taylor--Hood
elements in \cite{olshanskii2019inf}
suggests that $\rho_p \simeq h$ is the optimal choice.
 Therefore, \emph{in the remainder we restrict the stabilization parameters to}
\begin{equation} \label{choicerho}
\rho_u \simeq  h^{-1}, \quad \rho_p \simeq h, \quad \eta\simeq h^{-2}.
\end{equation}

\section{Well-posedness of discretizations} \label{sectwellposed}
Before we analyze the properties of the finite element bilinear forms we recall a lemma (\cite[Lemma 7.8]{grande2017higher} ) which shows that for finite element functions the $L^2$-norm in the neighborhood $\Omega_{\Theta}^{\Gamma}$ can be controlled by the $L^2$-norm on $\Gamma_h$ and the $L^2$-norm of the normal derivative  on $\Omega_{\Theta}^{\Gamma}$.
\begin{lemma} \label{lemmastabilizationinequality}
	For all $k \in \mathbb{N}$, $k\geq 1$, the following inequality holds:
	\begin{equation}\label{L2control}
	\Vert v_h \Vert_{L^2(\Omega_{\Theta}^{\Gamma})}^2 \lesssim h\Vert v_h \Vert_{L^2(\Gamma_h)}^2 + h^2\Vert \bn_h \cdot \nabla v_h \Vert_{L^2(\Omega_{\Theta}^{\Gamma})}^2   ~~~~\text{for all} ~v_h \in V_{h,\Theta}^{k}.
	\end{equation}
The result remains true if $\Omega_{\Theta}^{\Gamma},$ $\Gamma_h$, and $V_{h,\Theta}^{k}$ are replaced by $\Omega_{h}^{\Gamma},$ $\Gamma^{\rm lin}$, and $V_{h}^{k}$, respectively.
\end{lemma}

We formulate a few corollaries that are   useful in the remainder.
The following results are obtained by  application of \eqref{L2control}, \eqref{eqtraceestimate} and standard FE inverse inequalities:
%Using the trace inequality \eqref{eqtraceestimate} and a standard FE inverse inequality one obtains
\begin{align}
\Vert q_h \Vert_{L^2(\Omega_{\Theta}^{\Gamma})} & \simeq h^\frac12 \|q_h\|_M~~~\text{for all} ~q_h \in  V_{h,\Theta}^{k}, \label{HH3} \\
%\Vert v_h \Vert_{L^2(\Omega_{\Theta}^{\Gamma})}^2 &\simeq h\Vert v_h \Vert_{L^2(\Gamma_h)}^2 + h^2\Vert \bn_h \cdot \nabla v_h \Vert_{L^2(\Omega_{\Theta}^{\Gamma})}^2   ~~~\text{for all} ~v_h \in V_{h,\Theta}^{k}, \label{HH1} \\
\Vert \bv_h \Vert_{L^2(\Omega_{\Theta}^{\Gamma})} &\simeq h^\frac12 \Vert \bv_h \Vert_{L^2(\Gamma_h)} + h\Vert \nabla \bv_h \bn_h \Vert_{L^2(\Omega_{\Theta}^{\Gamma})}   ~~~\text{for all} ~\bv_h \in \bV_{h,\Theta}^{k}. \label{HH2}
\end{align}
Using \eqref{eqtraceestimate} and \eqref{L2control} we also obtain the surface inverse inequality
\begin{equation} \label{HH4}
\|\nabla q_h \|_{L^2(\Gamma_h)} \lesssim h^{-1} \|q_h\|_{L^2(\Gamma_h)} + h^{-\frac12} \|\bn_h \cdot \nabla q_h\|_{L^2(\Omega_{\Theta}^{\Gamma})} = h^{-1}\|q_h\|_M,~ ~q_h \in V_{h,\Theta}^{k},
\end{equation}
and the vector analog
\begin{equation} \label{HH4a}
\|\nabla \bv_h \|_{L^2(\Gamma_h)} \lesssim h^{-1} \|\bv_h\|_{L^2(\Gamma_h)} + h^{-\frac12} \| \nabla \bv_h\bn_h\|_{L^2(\Omega_{\Theta}^{\Gamma})},~~
\bv_h \in \bV_{h,\Theta}^{k}.
\end{equation}
\begin{lemma} \label{Ahscalarproduct}
	The following continuity and coercivity estimates hold:
	\begin{align}
	 & A_h(\bu,\bv)  \leq \Vert \bu \Vert_{A} \Vert \bv \Vert_{A},\quad
	b_h(\bu,q)  \lesssim \Vert \bu \Vert_{A} \Vert q \Vert_{M},\quad \forall~\bu, \bv \in \bV_{reg,h},~q \in V_{reg,h}, \label{EstCont}\\
	 & h^{-1} \Vert \bu_h \Vert_{L^2(\Omega_{\Theta}^{\Gamma})}^2  \lesssim A_h(\bu_h, \bu_h), \qquad
	\|\bu_h\|_{H^1(\Gamma_h)}^2 \lesssim A_h(\bu_h,\bu_h) \quad \forall~ \bu_h \in \bV_{h,\Theta}^{k}. \label{EstEll}
\end{align}
\end{lemma}
\begin{proof}
	The estimates in \eqref{EstCont} follow from the Cauchy-Schwarz inequality. The first result in \eqref{EstEll}  follows from \eqref{HH2}:
	\begin{equation*} \begin{split}
	A_h(\bu_h, \bu_h)  \geq    \Vert \bu_h \Vert_{L^2(\Gamma_h)}^2 + \rho_u \Vert \nabla \bu_h \bn_h \Vert_{L^2(\Omega_\Theta^\Gamma)}^2
	\gtrsim h^{-1} \Vert \bu_h \Vert_{L^2(\Omega_\Theta^{\Gamma})}^2.
	\end{split}
	\end{equation*}
	The second result in \eqref{EstEll} is proven in Lemma 5.16 in \cite{jankuhn2019trace}.
\end{proof}

The following \emph{inf-sup condition} is crucial for the well-posedness and error analysis of our FE formulations:
%\begin{assumption} \label{assumpinfsupcond}
	There exists $c_0>0$ independent of $h$ and the position of $\Gamma_h$ in the mesh such that
	\begin{equation} \label{infsup}
	c_0 \Vert q_h \Vert_M \leq \sup_{\bv_h \in \bU_h} \frac{|b_h(\bv_h,q_h)|}{\Vert \bv_h \Vert_A} + \tilde{s}_h(q_h,q_h)^\frac12 \qquad \forall q_h \in Q_h.
	\end{equation}
%\end{assumption}
Below we denote this condition by ``\emph{inf-sup condtion for} $\Gamma_h$''.

From the fact that $A(\cdot,\cdot)$ defines a scalar product on $\bU_h$, cf. Lemma~\ref{Ahscalarproduct}, and the inf-sup condition~\eqref{infsup} for $b_h(\cdot,\cdot)$ on $\bU_h \times Q_h$  it follows that problem \eqref{discreteform1} has a unique solution.

\subsection{Analysis of inf-sup condition for $\Gamma_h$} \label{sectstab}
In \cite{olshanskii2019inf}, an inf-sup  condition as in \eqref{infsup} was shown to hold (only) for $k=2$ and \emph{assuming exact integration of traces over} $\Gamma$, i.e. $\Gamma_h=\Gamma$. Below we show that the arguments can be extended to include the effect of geometric errors and to  $k\geq 2$.   The analysis of the effect of geometric errors on the stability properties of the trace Taylor--Hood pair, which has not been addressed in the literature so far, although rather technical, has a clear structure.  This strucure  is as follows. In the next section we derive an integration by part perturbation result for the bilinear form $b_h(\cdot,\cdot)$.  Using this result we then show, in section~\ref{sectRes2} that the inf-sup condition for $\Gamma_h$ follows from the analogous inf-sup condition for $\Gamma^{\rm lin}$. In section~\ref{sectRes3} we derive, using the inf-sup property of the bilinear form $b(\cdot,\cdot)$ for the pair $\bV_T \times L_0^2(\Gamma)$, an equivalent formulation of the inf-sup
condition for $\Gamma^{\rm lin}$ (``Verf\"urth's trick''). Finally, using results from \cite{olshanskii2019inf} a proof for $k\geq 2$ of this equivalent formulation of the inf-sup condition for $\Gamma^{\rm lin}$ is presented (section~\ref{sectRes4}).

\subsubsection{Integration by parts over $\Gamma_h$} \label{sectintpart}
On the smooth closed surface $\Gamma$  the partial integration rule $b(\bv,q)=- \int_{\Gamma} q \divG \bv \, ds= \int_{\Gamma} \bv \cdot \nabla_\Gamma q  \, ds$, for $\bv \in \bV_T$, $q \in H^1(\Gamma)$, holds. If
$\Gamma$ is replaced by $\Gamma_h$ or $\Gamma^{\rm lin}$ and we consider velocity fields that are not necessarily tangential,  extra terms arise due to jumps of co-normal vectors over edges.
Denote by $\mathcal{E}_h$ the collection of all edges in $\Gamma_h$. Let $E\in\mathcal{E}_h$ be the common edge of two surface segments $\Gamma_{T^+}, \Gamma_{T^-}\subset\Gamma_h$, and $\nu_h^+, \nu_h^-$ are the corresponding unit co-normals, i.e. $\nu_h^+$ is normal to $E$ and tangential for $\Gamma_{T^+}$.
Integration by parts over each smooth surface patch $\Gamma_T=\Gamma_h\cap \Theta(T)$,  leads to
\begin{multline}\label{aux1467}
\int_{\Gamma_h}\bv \cdot\nabla_{\Gamma_h}q\,ds_h\\ =-\int_{\Gamma_h}q\operatorname{div}_{\Gamma_h}\bv\,ds_h+\sum_{T\in\T_h^\Gamma}\int_{\Gamma_T}(\bv\cdot\bn_h)q\divGh\bn_h\,ds+
\sum_{E\in\mathcal{E}_h}\int_{E}[\nu_h\cdot\bv]q \,dl,
\end{multline}
for functions $\bv$, $q$ that are sufficiently smooth on each of the patches. An analogous formula holds with $\Gamma_h$ replaced by $\Gammalin$. Below, in Lemma~\ref{LemmapartInt} we derive a bound for the perturbation terms. As a preliminary result we derive trace results for $L^2$-norms on the set of edges $\mathcal{E}_h$.
\begin{lemma} \label{lemmaedgeu}
	The following trace inequalities hold:
	\begin{align}
	\|q_h\|_{L^2(\mathcal{E}_h)} & \lesssim h^{-\frac12} \Vert q_h \Vert_M \quad \text{for all}~q_h \in V_{h,\Theta}^{k},\label{eqedgexi} \\
	\|\bv_h\|_{L^2(\mathcal{E}_h)} & \lesssim h^{-\frac12} \Vert \bv_h \Vert_A \quad \text{for all}~\bv_h \in \bV_{h,\Theta}^{k}.\label{eqedgexiV}
	\end{align}
\end{lemma}

\begin{proof}
	Take $E\in \mathcal{E}_h$ and let $\Gamma_T\in\Gamma_h$ be a corresponding segment of which $E$ is an edge. Let $W$ be a side of the transformed tetrahedron $\Theta_h(T)$ such that $E \subset W$.
	We apply \eqref{eqtraceestimate}  and  a standard FE inverse inequality to obtain
	\begin{equation}\label{ERT} \begin{split}
	\int_E |q_h|^2 \, dl &\lesssim h^{-1} \Vert q_h \Vert^2_{L^2(W)} + h \Vert q_h \Vert^2_{H^1(W)} \lesssim h^{-1} \Vert q_h \Vert^2_{L^2(W)} \\
	& \lesssim h^{-2} \Vert q_h \Vert^2_{L^2(\Theta_h(T))} + \Vert q_h \Vert^2_{H^1(\Theta_h(T))} \lesssim h^{-2} \Vert q_h \Vert^2_{L^2(\Theta_h(T))}.
	\end{split}
       \end{equation}
	Summing over all edges and applying \eqref{L2control} completes the proof for \eqref{eqedgexi}. With very similar arguments, using \eqref{HH2}, one obtains the result \eqref{eqedgexiV}.
\end{proof}
\ \\

For $\bv_h\in\bV_k^k, q_h \in V_h^{k-1}$, we introduce the analogous $A$-norm and $M$-norm corresponding to the $\Gamma^{\rm lin}$  mesh:
\begin{align*} \|\bv_h \|_{A}^2 &:= \Vert \nabla_{\Gamma^{\rm lin}}\bv_h \Vert_{L^2(\Gamma^{\rm lin})}^2 + \Vert \bv_h \Vert_{L^2(\Gamma^{\rm lin})}^2+ \eta \Vert \bn_{\rm lin}\cdot\bv_h \Vert_{L^2(\Gamma^{\rm lin})}^2+\rho_u \Vert \nabla \bv_h \bn_{\rm lin} \Vert_{L^2(\Omega_h^\Gamma)}^2, \\
 \|q_h\|_M^2 & := \|q_h\|_{L^2(\Gammalin)}^2+\rho_p \|\bn_{\rm lin}\cdot \nabla q_h\|_{L^2(\OGamma)}^2.
\end{align*}

\begin{lemma} \label{LemmapartInt}The following estimates hold:
\begin{equation}
  \left|\int_{\Gamma_h}\bv_h \cdot\nabla_{\Gamma_h}q_h\,ds_h+\int_{\Gamma_h}q\operatorname{div}_{\Gamma_h}\bv\,ds_h\right|  \lesssim h\|\bv_h\|_{A}\|q_h\|_M, \label{b_diff}
\end{equation}
for all $\bv_h \in \bU_h,\, q\in Q_h$,
\begin{equation}
  \left|\int_{\Gammalin}\bv_h \cdot\nabla_{\Gammalin}q_h\,ds_h+\int_{\Gammalin}q\operatorname{div}_{\Gammalin}\bv\,ds_h\right| \lesssim h^\frac12 \|\bv_h\|_{A}\|q_h\|_M, \label{b_diffLin}
\end{equation}
for all $\bv_h \in \bV^k_h, \, q_h \in V^{k-1}_h$.
\end{lemma}
\begin{proof}
We use the identity \eqref{aux1467}. For the second term on the right-hand side in \eqref{aux1467} we use $\max\limits_{T\in\T_h^\Gamma}\|\divGh\bn_h\|_{L^\infty(\Gamma_T)}\lesssim1$, $\|\bn_h - \tilde \bn_h\|_{L^\infty(\Gamma_h)} \lesssim h$,  $\eta \sim h^{-2}$ and the definition of the $A$ and $M$ norms to get
\begin{align*}
& \left|\sum_{T\in\T_h^\Gamma}\int_{\Gamma_T}(\bv_h\cdot\bn_h)q_h\,\divGh\bn_h\,ds\right|\lesssim
\|\bv_h\cdot\bn_h\|_{L^2(\Gamma_h)} \|q_h\|_{L^2(\Gamma_h)} \\
& \lesssim \left(\|\bv_h\cdot \tilde \bn_h\|_{L^2(\Gamma_h)}+ h \|\bv_h\|_{L^2(\Gamma_h)}\right) \|q_h\|_M \lesssim h \|\bv_h\|_{A}\|q_h\|_M.
\end{align*}
The same holds for  $\Gammalin$ instead of $\Gamma_h$.
We now consider the third term on the right-hand side in \eqref{aux1467}.
For the surface approximation $\Gamma_h$ we have  $|[\nu_h]|\lesssim h^{k}$ with $k \geq 2$, and thus with $[\nu_h\cdot\bv_h]=\bv_h \cdot [\nu_h]$ we get, using Lemma~\ref{lemmaedgeu},
\[
 \left|\sum_{E\in\mathcal{E}_h}\int_{E}[\nu_h\cdot\bv_h] q_h\,dl\right| \lesssim h^2 \|\bv_h\|_{L^2(\mathcal{E}_h)}\|q_h\|_{L^2(\mathcal{E}_h)} \lesssim h \|\bv_h\|_A \|q_h\|_M.
\]
Combining these results, we obtain the estimate \eqref{b_diff}.
Finally we consider the third term on the right-hand side in \eqref{aux1467} for the case $\Gammalin$, which requires a more subtle treatment because we only have $|[\nu_h]|\lesssim h$. From Lemma~3.5 in~\cite{ORXimanum} we have (with $\mathcal{E}_h$ the set of edges in $\Gammalin$):
\begin{equation} \label{improvedest}
 \|\bP [\nu_h]\|_{L^\infty(\mathcal{E}_h)} \lesssim h^2.
\end{equation}
Given $E \in \mathcal{E}_h$, we split
\[
[\nu_h\cdot\bv_h]=[\nu_h] \bv_h= [\nu_h]\cdot\bP^{+}_{\rm lin}\bv_h+ [\nu_h]\cdot\bn_{\rm lin}^{+}(\bv_h\cdot\bn_{\rm lin}^{+}),
\] where $\bP^{+}_{\rm lin}$ and $\bn_{\rm lin}^{+}$ is the projector and normal to $\Gammalin$ from one (arbitrary chosen) side of the edge $E$.
%and the same arguments as in Lemma~\ref{lemmaprojectedconormaljump} (or just
Using  $|\bP_{\rm lin}-\bP|\lesssim h$, $|[\nu_h]|\lesssim h$  and \eqref{improvedest}, we get
\begin{equation}\label{aux660}
|[\nu_h]\cdot\bP^{+}_{\rm lin}\bv_h|=|\big(\bP^{+}_{\rm lin}-\bP)[\nu_h] + \bP[\nu_h]\big)\cdot\bv_h|\lesssim h^{2}|\bv_h|\quad\text{on}~E,
\end{equation}
and also
\begin{equation}\label{aux664}
| [\nu_h]\cdot\bn_{\rm lin}^{+}(\bv_h\cdot\bn_{\rm lin}^{+})|\lesssim h|\bv_h\cdot\bn_{\rm lin}^{+}|.
\end{equation}
%Now, applying \eqref{eqtraceestimate} in 2D and FE inverse inequalities we bound
%\[
%\int_{E}|\bv|^2 \,dl\lesssim h^{-2} \int_{T}|\bv|^2 \,dx\quad \text{for}~ T\in\T_h^\Gamma,~s.t.~E\in T.
%\]
Using \eqref{aux660}, \eqref{aux664} and the same arguments as in \eqref{ERT} we get
\begin{equation}\label{aux1466}
\left|\sum_{E\in\mathcal{E}_h}\int_{E}[\nu_h\cdot\bv_h] q_h\,dl \right| \lesssim \left( h^2 \|\bv_h\|_{L^2(\mathcal{E}_h)} + \|\bv_h \cdot \bn_{\rm lin}\|_{L^2(\OGamma)}\right) h^{-\frac12}\|q_h \|_M.
\end{equation}
We can approximate  the piecewise constant vector $\bn_{\rm lin}$ by $\hat \bn_h \in \bV_h^1$ such that $\|\bn_{\rm lin} - \hat \bn_h\|_{L^\infty(\Omega_h^\Gamma)} \lesssim h$ and $\|\nabla \hat \bn_{h}\|_{L^\infty(\Omega_h^\Gamma)} \lesssim 1$. Using this, triangle inequalities, \eqref{L2control} and \eqref{HH4a} we get
\begin{align*}
& \|\bv_h \cdot \bn_{\rm lin}\|_{L^2(\OGamma)}  \lesssim \|\bv_h \cdot \hat \bn_h\|_{L^2(\Omega_h^\Gamma)} +h \|\bv_h\|_{L^2(\Omega_h^\Gamma)} \\
 & \lesssim h^\frac12 \|\bv_h \cdot \bn_{\rm lin}\|_{L^2(\Gammalin)} + h^{\frac32} \|\bv_h\|_{L^2(\Gammalin)} + h \|\nabla \bv_h \bn_{\rm lin}\|_{L^2(\Omega_h^\Gamma)} \lesssim h \|\bv_h\|_A.
\end{align*}
Using this bound and the estimate \eqref{eqedgexiV} in \eqref{aux1466} completes the proof of \eqref{b_diffLin}.
\end{proof}

\subsubsection{Inf-sup condition for $\Gamma_h$ follows from inf-sup condition for $\Gamma^{\rm lin}$} \label{sectRes2}
\begin{lemma} Take  $k\ge2$.  For $h >0$ sufficiently small,  \eqref{infsup} follows from
\begin{equation} \label{infsuplin}
	\|q_h\|_{L^2(\Gamma^{\rm lin})} \lesssim \sup_{\bv_h \in \bV_h^k} \frac{ \int_{\Gammalin} \bv_h \cdot \nabla_{\Gammalin} q_h \, ds_h}{\|\bv_h \|_{A}} + h^{\frac12}\left\|\bn_{\rm lin}\cdot\nabla  q_h\right\|_{L^2(\Omega_h^{\Gamma})}~\forall q_h \in V^{k-1}_h.
	\end{equation}
\end{lemma}
\begin{proof}
 For $(\bu_h,r_h) \in \bU_h \times Q_h$,  we transform back to the piecewise polynomial functions: $\bu_h=\bv_h \circ (\Theta_h)^{-1}$, $\bv_h \in \bV_h^k$, $r_h= q_h \circ \Theta_h^{-1}$, $q_h \in V_h^{k-1}$. Using $|1-{\rm det}(D\Theta_h)| \lesssim h^2$ (change in surface measure) it follows that $\|q_h\|_{L^2(\Gammalin)} \sim \|r_h\|_{L^2(\Gamma_h)}$ holds. Using the change of variables, $\|\bn_{\rm lin}- \bn_h\|_{L^\infty(\Omega_h^{\Gamma})} \lesssim h$, a finite element inverse inequality and \eqref{HH3}, we estimate
\begin{equation}\label{aux725}
\begin{split}
  \|\bn_{\rm lin} \cdot \nabla q_h\|_{L^2(\Omega_h^{\Gamma})}&\lesssim \|\bn_h \cdot \nabla r_h\|_{L^2(\Omega_\Theta^{\Gamma})}+ ch\|\nabla r_h\|_{L^2(\Omega_\Theta^{\Gamma})}\\ &\lesssim  \|\bn_h \cdot \nabla r_h\|_{L^2(\Omega_\Theta^{\Gamma})}+ c h^{\frac12}\|r_h\|_{M}.
\end{split}
\end{equation}
Hence we get $\|q_h\|_M \lesssim \|r_h\|_M$ and with the same arguments $\|r_h\|_M \lesssim \|q_h\|_M$. Using $\|I- D\Theta_h\|_{\infty} \lesssim h$, $\|\bP_{\rm lin}- \bP_h\| \lesssim h$ and a discrete Korn inequality \cite[Lemma 5.16]{jankuhn2019trace} $\|\bv\|_{H^1(\Gamma^{\rm lin})} \lesssim \|\bv\|_A$ we obtain
\begin{equation}\label{aux723}
 \|\bu_h\|_A \lesssim \|\bv_h\|_A.
\end{equation}
 Thanks to \eqref{b_diff} we get
\[
b_h(\bu_h,r_h)= \int_{\Gamma_h} \bu_h\cdot\gradGh r_h \, ds_h \geq -\int_{\Gamma_h} r_h \divGh \bu_h \, ds_h - c h \|\bu_h\|_A \|r_h\|_{M}.
\]
Note that
\[
 {\Div}_{\Gammalin} \bv_h = {\rm tr}(\bP_{\rm lin} \nabla \bv_h \bP_{\rm lin})={\rm tr}(\bP_{\rm lin} D\Theta_h^T \nabla \bu_h \circ \Theta_h \bP_{\rm lin})
= \divGh \bu_h \circ \Theta_h  +E,\]
with $|E| \lesssim h \|\nabla \bu\|$. Using this and the discrete Korn's inequality yields
\[
\int_{\Gamma_h} r_h \divGh \bu_h \, ds_h = \int_{\Gamma^{\rm lin}} q_h {\rm div}_{\Gamma^{\rm lin}} \bv_h \, ds_h + \tilde E_h, \quad |\tilde E_h| \lesssim h \|\bu_h\|_A \|r_h\|_M.
\]
Using \eqref{b_diffLin} and \eqref{aux723} we thus obtain
\begin{align*}
& b(\bu_h,r_h)  \geq -\int_{\Gamma^{\rm lin}} q_h {\rm div}_{\Gamma^{\rm lin}} \bv_h \, ds_h - c h\|\bu_h\|_A \|r_h\|_{M}  \\
 & \geq  \int_{\Gamma^{\rm lin}} \bv_h \cdot \nabla_{\Gamma^{\rm lin}} q_h  \, ds_h- c h^\frac12\|\bv_h\|_A \|q_h\|_{M} - c h\|\bu_h\|_A \|r_h\|_{M}\\
 &  \geq  \int_{\Gamma^{\rm lin}} \bv_h \cdot  \nabla_{\Gamma^{\rm lin}} q_h  \, ds_h- c h^\frac12\|\bv_h\|_A \|r_h\|_{M}.
\end{align*}
Using this, \eqref{infsuplin} and \eqref{aux725} yields
\begin{align*}
 \|r_h\|_M & \simeq \|q_h\|_{M} \lesssim  \sup_{\bu_h \in \bU_h} \frac{ b(\bu_h,r_h)}{\|\bu_h \|_{A}} + h^{\frac12}\left(\int_{\Omega_h^{\Gamma}}|\bn_{\rm lin}\cdot\nabla  q_h|^2 \right)^\frac12 + h^\frac12 \|r_h\|_M  \\
 & \lesssim  \sup_{\bu_h \in \bU_h} \frac{ b(\bu_h,r_h)}{\|\bu_h \|_{A}} + h^{\frac12}\left(\int_{\Omega_\Theta^{\Gamma}}|\bn_{h}\cdot\nabla  r_h|^2 \right)^\frac12 + h^\frac12 \|r_h\|_M.
\end{align*}
Hence, for $h>0$ sufficiently small
 \eqref{infsup} holds.
\end{proof}

\subsubsection{Reformulation of inf-sup condition for $\Gamma^{\rm lin}$} \label{sectRes3}
We use a standard technique (Verf\"urth's trick) to derive a more convenient formulation of \eqref{infsuplin}. In this derivation the inf-sup property \eqref{LBB} of the continuous problem is used. There are some technical issues to deal with, because  \eqref{LBB} holds for $\Gamma$ and \eqref{infsuplin} is formulated with the approximation $\Gammalin$ of $\Gamma$.
We introduce $\|q_h\|_{1,h}^2:= \sum_{T \in \T_h^{\Gamma}} h_T \|\nabla q_h\|_{L^2(T)}^2$.
\begin{lemma}\label{lemmaequiv}
Take $k \geq 2$. The inf-sup condition for $\Gamma^{\rm lin}$ \eqref{infsuplin} is equivalent to
\begin{align}
  \|q_h\|_{1,\,h} & \lesssim \sup_{\bv_h \in \bV_h^k} \frac{\int_{\Gammalin} \bv_h \cdot \nabla_{\Gamma^{\rm lin}} q_h \, ds_h}{\|\bv_h \|_{A}} + h^\frac12\left\|\bn_{\rm lin}\cdot\nabla  q_h\right\|_{L^2(\Omega_h^{\Gamma})}~\forall~ q_h \in V^{k-1}_h.\label{LBB2A}
\end{align}
\end{lemma}
\begin{proof}
From a finite element inverse inequality and~\eqref{L2control} we get
 \[ \begin{split} \Big(\sum_{T \in \T_h^{\Gamma}} h_T \|\nabla q_h\|_{L^2(T)}^2\Big)^{\frac12}  & \lesssim h^{-\frac12} \|q_h\|_{L^2(\Omega_h^{\Gamma})} \\ & \lesssim \|q_h\|_{L^2(\Gammalin)} + h^\frac12\|\bn_{\rm lin}\cdot \nabla q_h\|_{L^2(\Omega_h^{\Gamma})} \quad \text{for all}~q_h \in V_h^{k-1}.
\end{split}\]
Hence, \eqref{infsuplin} implies \eqref{LBB2A}.

We now derive \eqref{LBB2A} $\Rightarrow$ \eqref{infsuplin}.
Consider $ q_h \in V_h^{k-1}$ and $q_h^\ell \in H^1(\Gamma)$, the  lifting of $q_h$ from $\Gamma^{\rm lin}$ to $\Gamma$. Thanks to the inf-sup property for the continuous problem, there exists
$\bv\in\bV_T$ such that
\begin{equation} \label{aux10A}
 \int_\Gamma \bv \cdot \gradG q_h^\ell \, ds  = \| q_h^\ell\|_{L^2(\Gamma)}^2 \quad\text{and}\quad  \|\bv\|_{H^1(\Gamma)} \lesssim \| q_h^\ell\|_{L^2(\Gamma)}\lesssim \| q_h\|_{L^2(\Gammalin)}.
\end{equation}
We consider  $\bv^e\in H^1(\mathcal{O}_h(\Gamma))$, a normal  extension of $\bv$ off the surface to a neighborhood $\mathcal{O}_h(\Gamma)$ of width $O(h)$ such that  $\Omega_h^{\Gamma}\subset\mathcal{O}_h(\Gamma)$.
Take $\bv_h\coloneqq  I_h ( \bv^e) \in \bV_h^2$, where $I_h : H^1(\mathcal{O}_h(\Gamma))^3\to \bV_h^2$ is the Cl\'{e}ment interpolation operator.
By standard arguments (see, e.g., \cite{reusken2015analysis}) based  on stability and approximation properties of $I_h ( \bv^e)$, one gets
\begin{equation} \label{aux1AA}
\begin{split}
 \|\bv_h\|_A^2&=\|I_h ( \bv^e)\|_A^2\\
~ & \lesssim \|I_h ( \bv^e)\|_{H^1(\Gamma^{\rm lin})}^2+h^{-2}\|I_h ( \bv^e)\cdot\bn^{\rm lin}\|_{L^2(\Gamma^{\rm lin})}^2+ h^{-1} \|\nabla (I_h ( \bv^e))\bn^{\rm lin}\|_{L^2(\OGamma)}^2\\
{\footnotesize  \eqref{eqtraceestimate}}~  &\lesssim \sum_{T\in \cT_h^\Gamma} h_T^{-1}\|I_h ( \bv^e)\|_{H^1(T)}^2+h^{-2}\|\big(I_h ( \bv^e)-\bv^e\big)\cdot \bn^{\rm lin}\|_{L^2(\Gamma^{\rm lin})}^2 \\
&\qquad+h^{-2}\|\bv^e\cdot(\bn^{\rm lin}-\bn)\|_{L^2(\Gamma^{\rm lin})}^2\\
{\footnotesize \eqref{eqtraceestimate}}~  &\lesssim \sum_{T\in \cT_h^\Gamma} h_T^{-1}\|\bv^e\|_{H^1(\omega(T))}^2+   h^{-2}  \sum_{T\in \cT_h^\Gamma} h_T^{-1}\|I_h ( \bv^e)-\bv^e\|_{L^2(T)}^2\\ &\qquad
 +\|\bv^e\|_{L^2(\Gamma^{\rm lin})}^2 +h^{-2}  \sum_{T\in \cT_h^\Gamma} h_T \|I_h ( \bv^e)-\bv^e\|_{H^1(T)}^2\\
 &\lesssim    \sum_{T\in \cT_h^\Gamma} h_T^{-1}\|\bv^e\|_{H^1(\omega(T))}^2+\|\bv^e\|_{L^2(\Gammalin)}^2\lesssim    h^{-1}\|\bv^e\|_{H^1(\OGamma)}^2+\|\bv\|_{L^2(\Gamma)}^2\\
{\footnotesize \eqref{lemmasobolevnormsneighborhood}}~  & \lesssim \|\bv\|_{H^1(\Gamma)}^2 \lesssim \|q_h\|_{L^2(\Gammalin)}^2.
\end{split}
\end{equation}

Using the trace inequality \eqref{eqtraceestimate} and approximation properties of $\bv_h=I_h ( \bv^e)$  one gets
\begin{equation}\label{aux1dA}
 \|\bv^e - \bv_h \|_{L^2(\Gammalin)}\lesssim h\|\bv\|_{H^1(\Gamma)}.
\end{equation}
We now consider the splitting
\begin{equation}\label{aux774}
 \int_{\Gammalin} \bv_h \cdot \nabla_{\Gammalin} q_h \, ds_h = \int_{\Gammalin} \bv^e \cdot\nabla_{\Gammalin} q_h \, ds_h +\int_{\Gammalin} (\bv_h -\bv^e) \cdot\nabla_{\Gammalin} q_h \, ds_h.
\end{equation}
The second term can be estimated using \eqref{aux1dA} and \eqref{aux10A}:
\[
  \left|\int_{\Gammalin} (\bv_h -\bv^e) \cdot\nabla_{\Gammalin} q_h \, ds_h \right| \lesssim \|q_h\|_{L^2(\Gammalin)} \|q_h\|_{1,h}.
\]
For lifting the first term from $\Gammalin$ to $\Gamma$ we use transformation rules, cf., e.g., \cite{Demlow06}:
\begin{align}
 \nabla_{\Gammalin} q_h(x) &= \bP_h(I-d\bH)\gradG q_h^\ell(p(x)),\quad  x \in \Gammalin, \label{trans1} \\
 \gradG q_h^\ell(p(x)) & = (I-d\bH)^{-1}(I- \frac{\bn \bn_{\rm lin}^T}{\bn_{\rm lin}^T \bn}) \gradGh q_h(x),  \quad x \in \Gammalin .\label{trans2}
\end{align}
The result \eqref{trans2} implies
\[
  \|\gradG q_h^\ell(p(\cdot))\|_{L^2(\Gammalin)} \lesssim \|\gradGh q_h\|_{L^2(\Gammalin)} \lesssim h^{-1} \|q_h\|_{1,h}.
\]
We treat the first term in \eqref{aux774} using perturbation arguments:
\[
\begin{split}
 & \int_{\Gammalin} \bv^e \cdot\nabla_{\Gammalin} q_h \, ds_h =
\int_{\Gammalin} \bv^e \cdot \bP_h(I-d\bH) \gradG q_h^\ell(p(\cdot)) \, ds_h \\
 & = \int_{\Gammalin} \bP\bP_h\bP \bv^e \cdot \gradG q_h^\ell(p(\cdot)) \, ds_h - \int_{\Gammalin} d \bH  \bP_h\bv^e \cdot \gradG q_h^\ell(p(\cdot)) \, ds_h \\
 & \geq \int_{\Gammalin} \bP\bP_h\bP \bv^e \cdot \gradG q_h^\ell(p(\cdot)) \, ds_h - c \|d\|_{L^\infty(\Gammalin)} \|\bv\|_{H^1(\Gamma)} h^{-1} \|q_h\|_{1,h} \\
& \geq \int_{\Gammalin} \bP\bP_h\bP \bv^e \cdot \gradG q_h^\ell(p(\cdot)) \, ds_h - c h \|\bv\|_{H^1(\Gamma)}  \|q_h\|_{1,h} \\
& = \int_{\Gammalin}  \bv^e \cdot \gradG q_h^\ell(p(\cdot)) \, ds_h +\int_{\Gammalin} (\bP\bP_h\bP -\bP)\bv^e \cdot \gradG q_h^\ell(p(\cdot)) \, ds_h - c h \|\bv\|_{H^1(\Gamma)}  \|q_h\|_{1,h} \\
 & \geq \int_{\Gammalin}  \bv^e \cdot \gradG q_h^\ell(p(\cdot)) \, ds_h - c h \|\bv\|_{H^1(\Gamma)}  \|q_h\|_{1,h} \\
 &= \int_\Gamma\bv \cdot\gradG q_h^\ell \, ds +\int_{\Gamma} (\mu_h^{-1}-1)\bv \cdot\gradG q_h^\ell \, ds- c h \|\bv\|_{H^1(\Gamma)}  \|q_h\|_{1,h} \\
 & \geq \int_\Gamma\bv \cdot\gradG q_h^\ell \, ds- c h \|\bv\|_{H^1(\Gamma)}  \|q_h\|_{1,h} = \|q_h^\ell\|^2_{L^2(\Gamma)}- c h \|\bv\|_{H^1(\Gamma)}  \|q_h\|_{1,h}\\
 & \gtrsim  (1-\tilde c h) \|q_h\|_{L^2(\Gammalin)}^2- c h \|\bv\|_{H^1(\Gamma)}  \|q_h\|_{1,h} \\ & \gtrsim  (1-\tilde c h) \|q_h\|_{L^2(\Gammalin)} \big(\|q_h\|_{L^2(\Gammalin)}- c h \|q_h\|_{1,h}\big).
\end{split}
\]
Take $h>0$ sufficiently small such that $1-\tilde c h >0$.
Dividing both sides of the above chain by $\|\bv_h\|_A$ and using \eqref{aux1AA} yields
\begin{equation}\label{aux552A}
 \| q_h\|_{L^2(\Gammalin)} - c  \|q_h\|_{1,\,h} \lesssim \sup_{\bv_h\in\bV_h^k}\frac{\int_{\Gammalin} \bv_h \cdot \nabla_{\Gammalin} q_h\, ds_h}{\|\bv_h\|_A}.
\end{equation}
From \eqref{LBB2A} and \eqref{aux552A} we have
\begin{align*}
  \|q_h\|_{L^2(\Gammalin)} & \lesssim \sup_{\bv_h\in\bV_h^k}\frac{\int_{\Gammalin} \bv_h \cdot \nabla_{\Gammalin} q_h\, ds_h}{\|\bv_h\|_A} + h^\frac12
  \left\|\bn_{\rm lin}\cdot\nabla  q_h\right\|_{L^2(\Omega_h^{\Gamma})},
\end{align*}
hence, \eqref{infsuplin} holds.
\end{proof}

\subsubsection{The inf-sup condition for $\Gamma^{\rm lin}$ holds for $k\geq 2$} \label{sectRes4}
In \cite{olshanskii2019inf} the alternative inf-sup condition~\eqref{LBB2A} was proved for $\mathbf{P}_2$--$P_1$ elements with the original smooth surface  $\Gamma$ instead of its approximation $\Gamma^{\rm lin}$. In this section we use arguments from that paper and analyze the inf-sup condition for $\Gamma^{\rm lin}$. We extend the analysis presented in \cite{olshanskii2019inf} in the sense that we show that the inf-sup condition~\eqref{LBB2A} (hence \eqref{infsuplin}) holds for \emph{all} $k\geq 2$.

For this analysis, as in \cite{olshanskii2019inf}, we derive a further condition that is equivalent to \eqref{LBB2A}, in which the norm on the left-hand side in \eqref{LBB2A} is replaced by a weaker one where $\sum_{T \in \cT_h^\Gamma}$ is replaced by $\sum_{T \in \cT^{\Gamma}_{\textnormal{reg}}}$ with
$\cT^{\Gamma}_{\textnormal{reg}} \subset  \cT_h^\Gamma$ a subset of ``regular elements.''
We define the set of \emph{regular elements} as  those $T \in \T_h^\Gamma$ for which the area of the intersection $\Gamma_T = \Gammalin \cap T$ is not less than
$\hat c_\cT h_T^2$ with some suitably chosen (cf.~\cite{olshanskii2019inf}) threshold  parameter $ \hat c_\cT > 0$:
\begin{equation} \label{defregGglobal}
  \cT^{\Gamma}_{\textnormal{reg}} \coloneqq \{\, T \in \cT_h^\Gamma\::\:|\Gamma_T| \geq \hat c_\cT h_T^2\,\}.
\end{equation}
We define a corresponding \emph{semi}norm on $V_h^{k-1}$:
\[
\|q\|_{1,\textnormal{reg}}\coloneqq \Big(\sum_{T \in \T_{\textnormal{reg}}^\Gamma} h_T \|\nabla q\|_{L^2(T)}^2\Big)^{\frac12}.
\]
The result in the following lemma is derived in \cite[Corollary 4.3]{olshanskii2019inf} for the case of the exact surface $\Gamma$. With very small modifications all arguments also apply if $\Gamma$ is replaced by $\Gammalin$.
\begin{lemma} \label{lemPO}
For $h>0$ sufficiently small the following holds:
\[
\|q_h\|_{1,h}^2 \lesssim \|q_h\|_{1,\textnormal{reg}}^2+ h \left\|\bn_{\rm lin}\cdot\nabla  q_h\right\|_{L^2(\Omega_h^{\Gamma})}^2  \quad \text{for all}~ q_h \in V^{k-1}_h.
\]
\end{lemma}
From this result and Lemma~\ref{lemmaequiv} we immediately obtain the following corollary.
\begin{corollary} \label{CorolLBB2AA}
The  inf-sup condition for $\Gamma^{\rm lin}$ \eqref{infsuplin} is equivalent to the following one:
 \begin{align}
  \|q_h\|_{1,\textnormal{reg}} & \lesssim \sup_{\bv_h \in \bV_h^k} \frac{\int_{\Gammalin} \bv_h \cdot \nabla_{\Gamma^{\rm lin}} q_h \, ds_h}{\|\bv_h \|_{A}} + h^\frac12 \left\|\bn_{\rm lin}\cdot\nabla  q_h\right\|_{L^2(\Omega_h^{\Gamma})} ~\forall~ q_h \in V^{k-1}_h.\label{LBB2AA}
\end{align}
\end{corollary}
We finally state the main stability  result.
\begin{theorem}
 Take $k\geq 2$. For $h>0$ sufficiently small the  inf-sup condition for $\Gamma^{\rm lin}$ \eqref{infsuplin} holds.
\end{theorem}
\begin{proof}
We show that condition \eqref{LBB2AA} is satisfied. Denote by  $\mathcal{E}_{\textnormal{reg}}$ the set of all edges of tetrahedra from  $\T^\Gamma_{\textnormal{reg}}$.
Let $\widetilde{\mathbf{t}}_E$ be a vector connecting the two endpoints of $E \in \mathcal{E}_{\textnormal{reg}}$ and $\mathbf{t}_E \coloneqq \widetilde{\mathbf{t}}_E/|\widetilde{\mathbf{t}}_E|$.  For each edge $E$ let  $\phi_E$ be the  quadratic nodal finite element  function corresponding to the midpoint of $E$.
For  $q\in V_h^{k-1}$, we define
\begin{equation}\label{VforQregA}
	\bv(\vect x) \coloneqq \sum_{E\in\mathcal{E}_{\textnormal{reg}}} h_E^2\phi_E(\vect x)\, [\mathbf{t}_E\cdot\nabla q(\vect x)]\mathbf{t}_E.
\end{equation}
This vector function is continuous on $\Omega$ and its components are piecewise polynomials of degree $k$, hence $\bv\in \bV_h^k$ holds.
Using $0 \leq \phi_E \leq 1$ in $T \in \T_h^{\Gamma}$, we obtain with $\Gamma_T= \Gammalin \cap T$,
%We compute using the definition of $\bv$, Cauchy-Schwarz inequality and $0\le\phi_e\le1$:
\begin{align}
	& (\bv,\nabla_{\Gammalin} q)_{L^2(\Gamma_T)}=(\bv,\Plin\nabla q)_{L^2(\Gamma_T)} \nonumber\\
	&= \int_{\Gamma_T}\sum_{E\in\mathcal{E}_{\textnormal{reg}}}h_E^2\phi_E\, |\Plin\mathbf{t}_E\cdot\nabla q|^2 \diff{s}+\int_{\Gamma_T}\sum_{E\in\mathcal{E}_{\textnormal{reg}}}h_E^2\phi_E\, (\Plin^\perp\mathbf{t}_E\cdot\nabla q) (\Plin\mathbf{t}_E\cdot\nabla q) \diff{s} \nonumber \\
	&\ge\frac12\int_{\Gamma_T} \sum_{E\in\mathcal{E}_{\textnormal{reg}}}h_E^2\phi_E\, |\Plin\mathbf{t}_E\cdot\nabla q|^2 \diff{s} -\frac12\int_{\Gamma_T}\sum_{E\in\mathcal{E}_{\textnormal{reg}}}h_E^2\phi_E\, |\Plin^\perp\mathbf{t}_E\cdot\nabla q|^2 \diff{s}\nonumber \\
	&\ge\frac12\int_{\Gamma_T} \sum_{E\in\mathcal{E}_{\textnormal{reg}}}h_E^2\phi_E\, |\Plin\mathbf{t}_E\cdot\nabla q|^2\diff{s} -\frac12\int_{\Gamma_T}\sum_{E\in\mathcal{E}(T)}h_E^2 |\bn_{\rm lin}\cdot\nabla q|^2\diff{s} \nonumber \\
	&\ge\frac12\int_{\Gamma_T} \sum_{E\in\mathcal{E}_{\textnormal{reg}}}h_E^2\phi_E\, |\Plin\mathbf{t}_E\cdot\nabla q|^2\diff{s} -3h_T^2 \|\bn_{\rm lin}\cdot\nabla q\|^2_{L^2(\Gamma_T)} \nonumber \\
	& \geq \frac12\int_{\Gamma_T} \sum_{E\in\mathcal{E}_{\textnormal{reg}}}h_E^2\phi_E\, |\Plin\mathbf{t}_E\cdot\nabla q|^2\diff{s}-c_1h_T\|\bn_{\rm lin}\cdot\nabla q\|_{L^2(T)}^2. \label{est1}
\end{align}
For the last inequality  we used the local trace inequality \eqref{eqtraceestimate} and a standard inverse estimate applied to the  piecewise  polynomial $\bn_{\rm lin}\cdot\nabla q$.
Hence, for every $T \in \T_h^{\Gamma}$ we have
\begin{equation} \label{est2}
	(\bv,\nabla_{\Gammalin} q)_{L^2(\Gamma_T)}+ c_1 h_T\|\bn_{\rm lin}\cdot\nabla q\|_{L^2(T)}^2 \geq 0.
\end{equation}
We now restrict to $T \in \cT_{\textnormal{reg}}^\Gamma$ and estimate the first term in \eqref{est1}.  Corresponding to $\Gamma_T=\Gammalin \cap T$ we define a so-called base face  $F_T$ of $T$ as that  face of $T$ with unit normal closest to the unit normal $\bn_{\rm lin}$ on $\Gamma_T$.
Using shape regularity of $\cT_h$, a transformation to the unit tetrahedron, equivalence of norms and \eqref{defregGglobal} it follows (cf.  \cite{olshanskii2019inf} for precise derivation) that there exists a surface segment  $\tilde \Gamma_T$ with  the following properties:
\begin{equation} \label{e3_}
 \widetilde\Gamma_T\subset\Gamma_T,\quad
|\widetilde\Gamma_T|\gtrsim h^2,\quad \phi_E\ge C>0\quad\text{on}~\widetilde\Gamma_T~~\text{for all}~E \subset F_T.
\end{equation}
 where the constant $C >0$ is independent of $h$ and of how $\Gamma_T$ intersects $T$. Note that for a polynomial $p$ of a fixed degree, we have
\begin{equation}\label{aux1234}
 \|p\|_{L^2(\Gamma_T)}\lesssim\|p\|_{L^2(\tilde \Gamma_T)} ,\quad \text{and}\quad
\|\nabla p\|_{L^2(T)}^2\lesssim \|\bn_{\rm lin}\cdot\nabla p\|_{L^2(T)}^2+h \|\nabla_\Gamma p\|_{L^2(\Gamma_T)}^2.
\end{equation}
To show the first estimate one may use standard arguments by inscribing a 2-ball of radius $\simeq h$ in $\widetilde\Gamma_T$, superscribing a 2-ball of radius $\simeq h$ around $\Gamma_T$, applying a mapping to a reference superscribed unit 2-ball and using  equivalence of norms in this reference domain. By a similar argument one shows the second inequality. Concerning the latter we note that with the unit 3-ball denoted by $\hat B_3$ and the planar segment $\hat P:=\hat B_3 \cap \{x_3=0\}$ the functional $p \to \|\frac{\partial p}{\partial x_3}\|_{L^2(\hat B_3)} + \|\frac{\partial p}{\partial x_1} +\frac{\partial p}{\partial x_2}\|_{L^2(\hat P)}$ defines a norm on the space of non-constant polynomials of a fixed degree.

Using the first estimate from \eqref{aux1234} and \eqref{e3_} we estimate the first term in \eqref{est1} as follows:
\[
\begin{split}
& \int_{\Gamma_T} \sum_{E\in\mathcal{E}_{\textnormal{reg}}}h_E^2\phi_E\, |\bP_{\rm lin}\mathbf{t}_E\cdot\nabla q|^2\diff{s}
  \gtrsim h_T^2 \sum_{E \subset F_T} \int_{\Gamma_T} \phi_E |\Plin\mathbf{t}_E\cdot\nabla q|^2\diff{s} \\
    & \gtrsim h_T^2 \sum_{E \subset F_T} \int_{\widetilde\Gamma_T} \phi_E |\Plin\mathbf{t}_E\cdot\nabla q|^2\diff{s}
     \gtrsim h_T^2 \sum_{E \subset F_T} \int_{\widetilde\Gamma_T} |\Plin\mathbf{t}_E\cdot\nabla q|^2\diff{s}  \\
    & \gtrsim h_T^2 \sum_{E \subset F_T} \int_{\Gamma_T} |\Plin\mathbf{t}_E\cdot\nabla q|^2\diff{s} .
\end{split}
\]
Due to the construction of the base face $F_T$   we have  that  $|\bn_{\rm lin} \cdot \bn_{F_T}| $ is uniformly bounded away from zero. This implies that for any $\bz \in \Bbb{R}^3$ we have $\sum_{E \subset F_T}|\Plin \mathbf{t}_E\cdot\bz |^2 =\sum_{E \subset F_T}| \mathbf{t}_E\cdot\Plin\bz|^2 \gtrsim |\Plin \bz|^2$. Using this
and the second inequality in \eqref{aux1234} %that $\Gamma_T$ is a regular
we get
\[
 \begin{split}
   \int_{\Gamma_T} \sum_{E\in\mathcal{E}_{\textnormal{reg}}}h_E^2\phi_E  |\Plin\mathbf{t}_E\cdot\nabla q|^2\diff{s} &
  \gtrsim  h_T^2 \int_{\Gamma_T} |\Plin\nabla q|^2\diff{s} = h_T^2 \int_{\Gamma_T} |\nabla_{\Gammalin} q|^2\diff{s} \\ &\hskip-2ex \gtrsim  h_T \|\nabla q\|_{L^2(T)}^2- h_T \|\bn_{\rm lin}\cdot \nabla q\|_{L^2(T)}^2\\
 &\hskip-2ex \gtrsim h_T \|\nabla q\|_{L^2(T)}^2- h_T \|\bn\cdot \nabla q\|_{L^2(T)}^2 -c h_T^2 \|\nabla q\|_{L^2(T)}^2.
 \end{split}
\]
Substituting this in \eqref{est1} we obtain for $T \in \cT_{\textnormal{reg}}^\Gamma$:
\begin{equation} \label{E6}
  (\bv,\nabla_{\Gammalin} q)_{L^2(\Gamma_T)}+ c\, h_T\|\bn_{\rm lin}\cdot\nabla q\|_{L^2(T)}^2
\gtrsim   h_T \|\nabla q\|_{L^2(T)}^2.
\end{equation}
Combining this with \eqref{est2} and summing over $T \in \cT_h^\Gamma$ yields
\begin{equation} \label{E7}
 \int_{\Gammalin}\bv \cdot \nabla_{\Gammalin} q \, ds_h+ c\, h \|\bn_{\rm lin}\cdot\nabla q\|_{L^2(\Omega_h^\Gamma)}^2 \gtrsim \|q\|_{1,\textnormal{reg}}^2
\end{equation}
We use the following elementary observation: For positive numbers $\alpha,\beta,\delta$ the inequality $\alpha + \beta^2 \geq c_0 \delta^2$ implies $\alpha + \beta(\beta +\delta) \geq \min\{c_0,1\} \delta(\beta +\delta)$ and thus $\frac{\alpha}{\beta +\delta} + \beta \geq \min\{c_0,1\} \delta$. Therefore, estimate \eqref{E7} implies
\begin{equation}\label{E8}
\frac{\int_{\Gammalin}\bv \cdot \nabla_{\Gammalin} q \, ds_h}{ \|q\|_{1,\textnormal{reg}}+h^\frac12 \|\bn_{\rm lin}\cdot\nabla q\|_{L^2(\Omega_h^\Gamma)}} + h^\frac12\|\bn_{\rm lin}\cdot\nabla q\|_{L^2(\Omega_h^\Gamma)}  \gtrsim \|q\|_{1,\textnormal{reg}}.
\end{equation}
It remains to estimate $\|\bv\|_A$. Straightforward estimates (cf. details in \cite{olshanskii2019inf}) yield
\begin{align*}
 \|\nabla_{\Gammalin}\bv\|^2_{L^2(\Gammalin)} +\|\bv\|^2_{L^2(\Gammalin)} & \lesssim  \|q\|_{1,h}^2, \\
 \eta \|\bn_{\rm lin} \cdot \bv\|_{L^2(\Gammalin)}^2 & \simeq h^{-2} \|\bn_{\rm lin} \cdot \bv\|_{L^2(\Gammalin)}^2 \lesssim \|q\|_{1,h}^2, \\
 \rho_u \|\nabla \bv \, \bn_{\rm lin}\|_{L^2(\Omega_h^\Gamma)}^2 & \simeq h^{-1}\|\nabla \bv\,  \bn_{\rm lin}\|_{L^2(\Omega_h^\Gamma)}^2 \lesssim \|q\|_{1,h}^2.
\end{align*}
This yields $\|\bv\|_A \lesssim \|q\|_{1,h}$, and using Lemma~\ref{lemPO} we  get
\[
 \|\bv\|_A \lesssim   \|q\|_{1,\textnormal{reg}}+h^\frac12 \left\|\bn_{\rm lin}\cdot\nabla  q \right\|_{L^2(\Omega_h^{\Gamma})}.
\]
Combining this with \eqref{E8} completes the proof.
\end{proof}

\section{Error analysis} \label{sectAnalysis}

%\subsection{Strang-Lemma}

As usual, the discretization error analysis is based on a Strang type Lemma which bounds the discretization error in terms of an approximation error and a consistency error.
We define the bilinear form
\begin{equation} \label{defcA}
\mathcal{A}_h((\bu, p),(\bv, q)) := A_h(\bu,\bv) + b_h(\bv, p) + b_h(\bu, q) - \tilde{s}_h(p,q),
\end{equation}
for $(\bu, p),(\bv, q) \in \bV_{reg,h} \times V_{reg,h}$. Stability of the discrete problem \eqref{discreteform1}, uniformly in $h$ and the position of $\Gamma$ in the triangulation, in the product norm $\|\cdot\|_A \times \|\cdot\|_M$  follows from the  inf-sup property \eqref{infsup}.  Hence, for $\mathcal{A}_h( \cdot,\cdot)$ it holds,
\begin{equation} \label{eqinfsupAh}
\sup_{(\bv_h,q_h)  \in \bU_h \times Q_h} \frac{\mathcal{A}_h((\bu_h,p_h), (\bv_h,q_h)) }{\left(\Vert \bv_h \Vert_{A}^2 + \Vert q_h \Vert_{M}^2\right)^\frac{1}{2}} \gtrsim \left(\Vert \bu_h \Vert_{A}^2 + \Vert p_h \Vert_{M}^2\right)^\frac{1}{2},
\end{equation}
for all $(\bu_h,p_h) \in \bU_h \times Q_h$. This and the continuity of the $\mathcal{A}_h$ form yield the following Strang's-type Lemma. Here and in the remainder we use that  the solution $(\bu, p)  \in \bV_T \times L^2_0(\Gamma)$ of \eqref{contform} is sufficiently regular, in particular $(\bu, p)\in \bV_{reg,h} \times V_{reg,h}$.

\begin{lemma}[Strang's Lemma] \label{stranglemma}
Let $(\bu, p)  \in \bV_T \times L^2_0(\Gamma)$ be the unique solution of problem \eqref{contform} and $(\bu_h, p_h) \in \bU_h \times Q_h$ the unique solution of the finite element problem \eqref{discreteform1} . The following discretization error bound holds:
\begin{multline}\label{eqStrang}
\Vert \bu^e - \bu_{h} \Vert_{A} + \Vert p^e - p_{h} \Vert_{M}
\lesssim \min_{(\bv_h,q_h)  \in \bU_h \times Q_h} \left( \Vert \bu^e - \bv_h \Vert_{A} + \Vert p^e - q_{h} \Vert_{M}\right) \\
+ \sup_{(\bv_h,q_h)  \in \bU_h \times Q_h} \frac{\vert \mathcal{A}_h((\bu^e,p^e), (\bv_h,q_h)) - (\bbf_h, \bv_h)_{L^2(\Gamma_h)} + (g_h, q_h)_{L^2(\Gamma_h)} \vert}{\left(\Vert \bv_h \Vert_{A}^2 + \Vert q_h \Vert_{M}^2\right)^\frac{1}{2}}.
\end{multline}
\end{lemma}

%\subsection{Approximation error bounds}
The following lemma deals with the approximation error bounds in the norms that occur in the Strang lemma above. A proof can be found in \cite[Lemma 5.10]{jankuhn2019higher}.
\begin{lemma}[Approximation bounds] \label{lemmaapproximationerror}
For $\bu \in H^{k+1}(\Gamma)^3$ and $p \in H^{k}(\Gamma)$  the following approximation error bounds hold:
\begin{multline} \label{Eq2}
\min_{(\bv_h,q_h) \in \bU_h \times Q_h} \left( \Vert \bu^e - \bv_h \Vert_{A} + \Vert p^e - q_h \Vert_{M} \right)
\lesssim   h^{k} \left(\Vert \bu \Vert_{H^{k+1}(\Gamma)} + \Vert p \Vert_{H^{k}(\Gamma)}\right).
\end{multline}
\end{lemma}
\subsection{Consistency error analysis} The goal of this section is to provide an estimate of the consistency term on the right-hand side of \eqref{eqStrang}. We will use results obtained for a vector-Laplace problem in  \cite{jankuhn2019higher}. The variatonal formulation of that vector-Laplace problem results in a bilinear form that is the same as the $A_h(\cdot,\cdot)$ bilinear form, which is part of $\mathcal{A}_h(\cdot,\cdot)$ in \eqref{defcA}.

\subsubsection{Preliminaries} \label{preliminaries}
We start with results concerning the transformation of the integrals between $\Gamma$ and $\Gamma_h$. Using that the gradient of the closest point projection is given by $\nabla p = \bP- d \bH$, one computes for $u \in H^1(\Gamma)$ and  $x \in \Gamma_h$
\begin{equation} \label{transfo1} \gradGh u^e(x) %&= \gradGh (u \circ p)(x) = \bP_h(x) \nabla p(x) \nabla u(p(x)) \\
%&= \bP_h(x) (\bP(x) - d(x)\bH(x))\nabla u(p(x))
 = \bB^T(x) \gradG u(p(x)),~~  \text{with}~ \bB=\bB(x) := \bP(\bI - d\bH)\bP_h.
\end{equation}

The following properties of $\bB$ are known in the literature~\cite{hansbo2016analysis}:
\begin{lemma} \label{lemmaB}
For $x \in \Gamma_h$ and $\bB=\bB(x)$ as above, the map $\bB$ is invertible on the range of $\bP$ for $h$ small enough, i.e. there is $\bB^{-1} \colon {\rm range} (\bP(x)) \to {\rm range} (\bP_h(x))$ such that
%\begin{equation*}
$\bB\bB^{-1} = \bP, \, \bB^{-1} \bB = \bP_h$,
%\end{equation*}
and we have for $u \in H^1(\Gamma)$, $x \in \Gamma_h$,
\begin{equation*}
\gradG u(p(x)) = \bP(x)\bB^{-T}(x)\gradGh u^e(x).
\end{equation*}
Furthermore, the following estimates hold:
\[
\begin{split}
\Vert \bB \Vert_{L^\infty(\Gamma_h)} + \Vert \bP_h\bB^{-1} \bP \Vert_{L^\infty(\Gamma_h)} &\lesssim 1,\\\
\Vert \bP \bP_h - \bB \Vert_{L^\infty(\Gamma_h)}+\Vert \bP_h \bP - \bP_h\bB^{-1}\bP \Vert_{L^\infty(\Gamma_h)} &\lesssim h^{k+1}.
\end{split}
\]
For the surface measures on $\Gamma$ and $\Gamma_h$  the identity
%\begin{equation*}
$
d\Gamma = \vert \bB \vert d\Gamma_h
%\end{equation*}
$ holds,
with $\vert \bB \vert = \vert det(\bB) \vert$, and we have the estimates
\begin{equation*}
\Vert 1- \vert \bB \vert \Vert_{L^{\infty}(\Gamma_h)} \lesssim h^{k+1}, \quad \Vert \vert \bB \vert \Vert_{L^{\infty}(\Gamma_h)} \lesssim 1 , \quad \Vert \vert \bB \vert^{-1} \Vert_{L^{\infty}(\Gamma_h)} \lesssim 1.
\end{equation*}
\end{lemma}
Applying Lemma \ref{lemmaB} yields, for $u \in H^1(\Gamma)$,
\begin{equation*}
\gradG u^l(p(x)) = \bP(x)\bB^{-T}(x)\gradGh u(x), \quad x \in \Gamma_h.
\end{equation*}
Similar useful transformation results for vector-valued functions are given in the following corollary  from \cite{jankuhn2019higher}:
\begin{corollary} \label{corollarygradients}
For $\bu \in H^1(\Gamma)^3$ and $\bv \in H^1(\Gamma_h)^3$ we have
\begin{equation*} \begin{split}
\left(\nabla \bu^e \bP\right)^e &= \nabla \bu^e \bP = \nabla \bu^e \bP_h \bB^{-1}\bP \quad \text{on } \Gamma_h, \\
\left(\nabla \bv^l \bP \right)^e &= \nabla \bv^l \bP = \nabla \bv^l \bP_h \bB^{-1} \bP \quad \text{on } \Gamma_h.
\end{split}
\end{equation*}
\end{corollary}

\subsubsection{Consistency error bounds} \label{sectgeometry} We are now prepared to estimate the last term on the right-hand side of \eqref{eqStrang}.
We introduce further notation.
We define, for $\bv, \bw \in \bV_{reg,h}$, $q \in V_{reg,h}$:
\begin{align*}
 G (\bv,\bw) &:= a_{h}(\bv,\bw) - a(\bP\bv^l,\bP \bw^l) + s_h(\bv,\bw) + k_h(\bv,\bw), \\
 G_{b}(\bv,q) &:= b_h(\bv, q) - b(\bP\bv^l, q^l), \qquad G_f(\bw)   := (\bbf,\bw^l)_{L^2(\Gamma)} - (\bbf_h,\bw)_{L^2(\Gamma_h)}, \\
  G_g(q)   &:= (g_h,q)_{L^2(\Gamma_h)} - (g,q^l)_{L^2(\Gamma)}.
\end{align*}
Let $(\bu, p)  \in \bV_T \times L^2_0(\Gamma)$ be the unique solution of problem \eqref{contform} and $(\bv_h,q_h) \in \bU_h \times Q_h$.
The consistency term in \eqref{eqStrang} can be written as
\begin{equation} \label{consterm1} \begin{split}
&\mathcal{A}_h((\bu^e,p^e), (\bv_h,q_h)) - (\bbf_h, \bv_h)_{L^2(\Gamma_h)} + (g_h, q_h)_{L^2(\Gamma_h)}  \\
& = A_h(\bu^e, \bv_h) + b_h(\bv_h,p^e) + b_h(\bu^e, q_h) -  \tilde{s}_h(p^e,q_h) - (\bbf_h, \bv_h)_{L^2(\Gamma_h)} + (g_h, q_h)_{L^2(\Gamma_h)}  \\
&\qquad + \underbrace{(\bbf, \bv_h^l)_{L^2(\Gamma)} -  (g, q_h^l)_{L^2(\Gamma)}  -  a(\bu,\bP\bv_h^l) - b(\bP\bv_h^l, p) - b(\bu, q_h^l)}_{=0} \\
&= G (\bu^e, \bv_h) + G_b(\bv_h,p^e) + G_b(\bu^e, q_h) -  \tilde{s}_h(p^e,q_h) + G_f(\bv_h) + G_g(q_h).
\end{split}
\end{equation}
In \cite[Lemma 5.15, 5.18]{jankuhn2019higher} several $G$-terms in \eqref{consterm1} have already been analyzed. We collect these results in the following lemma.

\begin{lemma}\label{lemconslaplace}
Let $\bbf_h$ and $g_h$ be approximations of $\bbf$ and $g$ such that $\Vert \vert \bB\vert \bbf^e - \bbf_h \Vert_{L^2(\Gamma_h)} \lesssim h^{k+1} \Vert \bbf \Vert_{L^2(\Gamma)}$ and $\Vert \vert \bB\vert g^e - g_h \Vert_{L^2(\Gamma_h)} \lesssim h^{k+1} \Vert g \Vert_{L^2(\Gamma)}$. For the unique solution $(\bu, p)  \in \bV_T \times L^2_0(\Gamma)$ of problem \eqref{contform} and for all $(\bv_h,q_h) \in \bU_h \times Q_h$ the following holds:
\begin{equation*} \begin{split}
\vert G (\bu^e, \bv_h) \vert &\lesssim h^{k}\Vert \bu \Vert_{H^1(\Gamma)}\Vert \bv_h \Vert_{A}, \quad
\vert \tilde{s}_h(p^e,q_h) \vert \lesssim h^{k} \Vert p \Vert_{H^1(\Gamma)} \Vert q_h \Vert_M, \\
\vert G_f(\bv_h) \vert &\lesssim h^{k+1} \Vert \bbf \Vert_{L^2(\Gamma)} \Vert \bv_h \Vert_{L^2(\Gamma_h)}, \quad
\vert G_g(q_h) \vert \lesssim h^{k+1} \Vert g \Vert_{L^2(\Gamma)} \Vert q_h \Vert_{L^2(\Gamma_h)}.
\end{split}
\end{equation*}
\end{lemma}

The two terms left to be analyzed are $G_b(\bv_h,p^e)$ and $G_b(\bu^e, q_h)$, which result from geometric inconsistencies due to the difference in the bilinear forms $b(\cdot,\cdot)$ and $b_h(\cdot,\cdot)$.  A bound for $G_b(\bv_h,p^e)$ can be easily derived using Lemma \ref{lemmaB}.  For the term $G_b(\bu^e, q_h)$, however, we need to locally apply Green's formula.

\begin{lemma} \label{lemmabdifference}
Let $(\bu, p)  \in \bV_T \times L^2_0(\Gamma)$ be the unique solution of \eqref{contform} and assume that $(\bu, p)  \in H^2(\Gamma)^3 \times H^1(\Gamma)$. Then for all $(\bv_h, q_h) \in \bU_h \times Q_h$  the following holds:
\begin{equation}\label{eqbdiff1}
\vert G_b(\bv_h,p^e) \vert \lesssim  h^{k} \Vert \bv_h \Vert_{A} \Vert p \Vert_{H^1(\Gamma)}, \quad
\vert G_b(\bu^e, q_h) \vert \lesssim h^{k} \Vert \bu \Vert_{H^2(\Gamma)} \Vert q_h \Vert_{M}.
\end{equation}
\end{lemma}
\begin{proof}
 For the first estimate  we use Lemma \ref{lemmaB} and get
\begin{equation*} \begin{split}
&\vert G_b(\bv_h,p^e) \vert =\vert b_h(\bv_h,p^e) - b(\bP\bv_h^l, p) \vert = \vert (\bv_h , \gradGh p^e)_{L^2(\Gamma_h)} - (\vert \bB \vert \bv_h , (\gradG p)^e)_{L^2(\Gamma_h)} \vert  \\
&= \vert (\bv_h , (\bP_h - \bP)\gradGh p^e)_{L^2(\Gamma_h)} + (\bv_h , \bP\bP_h\gradGh p^e)_{L^2(\Gamma_h)} \\
&\hspace*{4.8mm}- (\bv_h , \bP\bB^{-T}\bP_h\gradGh p^e)_{L^2(\Gamma_h)}  + ((1-\vert \bB \vert)\bv_h , \bP\bB^{-T}\gradGh p^e)_{L^2(\Gamma_h)} \vert \\
&\lesssim  \big(\Vert \bP_h - \bP \Vert_{L^{\infty}(\Gamma_h)} + \Vert \bP\bP_h - \bP\bB^{-T}\bP_h \Vert_{L^{\infty}(\Gamma_h)} + \Vert 1-\vert \bB \vert \Vert_{L^{\infty}(\Gamma_h)}\big) \Vert \bv_h \Vert_{L^2(\Gamma_h)} \Vert p \Vert_{H^1(\Gamma)}\\
&\lesssim  h^{k} \Vert \bv_h \Vert_{A} \Vert p \Vert_{H^1(\Gamma)}.
\end{split}
\end{equation*}
We now consider the second estimate  in \eqref{eqbdiff1}. We use Green's formula and Lemma~\ref{lemmaB}, and thus obtain, with $\mathcal{E}_h$ and
$\nu_h$ as in section~\ref{sectintpart}:
\begin{equation} \label{eqb1} \begin{split}
G_b(\bu^e, q_h) &= b_h(\bu^e, q_h) - b(\bP\bu, q_h^l) \\
&= \sum_{T \in \mathcal{T}_h^{\Gamma}} (\bu^e , \gradGh q_h)_{L^2(\Gamma_T)} - (\bu , \gradG q_h^l)_{L^2(\Gamma)} \\
&= \sum_{T \in \mathcal{T}_h^{\Gamma}} -(\divGh (\bP_h \bu^e) , q_h)_{L^2(\Gamma_T)} + \sum_{E \in \mathcal{E}_h} ([\nu_h]\cdot \bu^e,q_h)_{L^2(E)} \\
&\hspace*{4.8mm}+ (\divG \bu , q_h^l)_{L^2(\Gamma)} \\
&= \sum_{T \in \mathcal{T}_h^{\Gamma}} -(\divGh (\bP_h \bu^e) , q_h)_{L^2(\Gamma_T)} + \sum_{E \in \mathcal{E}_h} ([\nu_h]\cdot \bu^e,q_h)_{L^2(E)} \\
&\hspace*{4.8mm}+ (\vert \bB \vert (\divG \bu)^e , q_h)_{L^2(\Gamma_h)}.
\end{split}
\end{equation}
Note that $\divGh (\bP_h \bu^e) = \divGh \bu^e - (\bu^e \cdot \bn_h) \tr(\gradGh \bn_h)$ on $\Gamma_T$ and $(\divG \bu)^e = \tr((\bP\nabla \bu^e \bP)^e) = \tr(\bP\nabla \bu^e \bP_h \bB^{-1}\bP)$ on $\Gamma_h$ (by Lemma \ref{lemmaB}) holds. Hence, we have
\begin{equation*} \begin{split}
&\vert \bB \vert (\divG \bu)^e - \divGh \bu^e \\
&= \vert \bB \vert \tr(\bP\nabla \bu^e \bP_h \bB^{-1}\bP) - \tr(\bP_h \nabla \bu^e \bP_h) \\
&= (\vert \bB \vert - 1) \tr(\bP\nabla \bu^e \bP_h \bB^{-1}\bP) +  \tr(\bP\nabla \bu^e \bP_h \bB^{-1}\bP)  - \tr((\bP_h - \bP) \nabla \bu^e \bP_h) \\
&\hspace*{4.8mm}+ \tr( \bP \nabla \bu^e \bP_h (\bP - \bP_h)) - \tr( \bP \nabla \bu^e \bP_h \bP) \\
&= (\vert \bB \vert - 1) \tr(\bP\nabla \bu^e \bP_h \bB^{-1}\bP) +  \tr(\bP\nabla \bu^e (\bP_h \bB^{-1}\bP - \bP_h \bP)) \\
&\hspace*{4.8mm}- \tr((\bP_h - \bP) \nabla \bu^e \bP_h)+ \tr( \bP \nabla \bu^e \bP_h (\bP - \bP_h)).
\end{split}
\end{equation*}
Therefore, using Lemma \ref{lemmaB} we obtain for the sum of the first and last term on the right-hand side of equation \eqref{eqb1}
\begin{equation*} \begin{split}
&\Big|\sum_{T \in \mathcal{T}_h^{\Gamma}} -(\divGh (\bP_h \bu^e) , q_h)_{L^2(\Gamma_T)} + (\vert \bB \vert (\divG \bu)^e , q_h)_{L^2(\Gamma_h)} \Big| \\
&= \Big| \sum_{T \in \mathcal{T}_h^{\Gamma}} \Big(  -(\divGh \bu^e , q_h)_{L^2(\Gamma_T)} + (\vert \bB \vert (\divG \bu)^e , q_h)_{L^2(\Gamma_T)} \\
&\hspace*{4.8mm}+ ((\bu^e \cdot \bn_h ) \tr(\gradGh \bn_h) , q_h)_{L^2(\Gamma_T)}\Big) \Big|\\
&\lesssim \vert  (\vert \bB \vert (\divG \bu)^e - \divGh \bu^e , q_h)_{L^2(\Gamma_h)} \vert + \sum_{T \in \mathcal{T}_h^{\Gamma}} \vert ((\bu^e \cdot (\bn_h - \bn)) \tr(\gradGh \bn_h) , q_h)_{L^2(\Gamma_T)}\vert \\
&\lesssim (\Vert 1-\vert \bB \vert \Vert_{L^{\infty}(\Gamma_h)} + \Vert \bP_h\bB^{-1}\bP - \bP_h\bP \Vert_{L^{\infty}(\Gamma_h)} + \Vert \bP_h - \bP \Vert_{L^{\infty}(\Gamma_h)}) \Vert \bu \Vert_{H^1(\Gamma)} \Vert q_h \Vert_{M} \\
&\hspace*{4.8mm}+ \sum_{T \in \mathcal{T}_h^{\Gamma}} \Vert \bn_h -\bn  \Vert_{L^{\infty}(\Gamma_T)} \Vert \bu^e \Vert_{L^2(\Gamma_T)} \Vert q_h\Vert_{L^2(\Gamma_T)} \\
&\lesssim h^{k} \Vert \bu \Vert_{H^1(\Gamma)} \Vert q_h \Vert_{M} + \sum_{T \in \mathcal{T}_h^{\Gamma}} h^{k } \Vert \bu^e \Vert_{L^2(\Gamma_T)} \Vert q_h\Vert_{L^2(\Gamma_T)} \\
&\lesssim h^{k} \Vert \bu \Vert_{H^1(\Gamma)} \Vert q_h \Vert_{M}.
\end{split}
\end{equation*}
For the second term on the right-hand side of equation \eqref{eqb1} we  need a bound on the jump in the conormals across the edges $E$. Such a bound is derived in \cite[Lemma 3.5]{ORXimanum} for the case of a piecewise planar surface approximation. The arguments immediately extend to the higher order surface approximation $\Gamma_h$, resulting in the estimate
\begin{equation*}
	\Vert \bP [\nu_h]\|_{L^\infty(\mathcal{E}_h)} \lesssim h^{2k}.
\end{equation*}
Using \eqref{eqtraceestimate} and arguments similar to \eqref{ERT} we get
\begin{equation*}
 \|\bu^e\|_{L^2(\mathcal{E}_h)} \lesssim h^{-1} \|\bu^e\|_{H^2(\Omega_\Theta^\Gamma)} \lesssim h^{-\frac12} \|\bu\|_{H^2(\Gamma)}.
\end{equation*}
Using these estimates and the result \eqref{eqedgexi} we obtain
\begin{equation*}
\sum_{E \in \mathcal{E}_h} ([\nu_h]\cdot \bu^e,q_h)_{L^2(E)} \lesssim \|\bP [\nu_h]\|_{L^\infty(\mathcal{E}_h)} \|\bu^e\|_{L^2(\mathcal{E}_h)}   \Vert q_h\Vert_{L^2(\mathcal{E}_h)}
\lesssim h^{2k-1}  \Vert \bu \Vert_{H^2(\Gamma)} \Vert q_h \Vert_M,
\end{equation*}
which completes the proof for the second estimate in \eqref{eqbdiff1}.
\end{proof}

Applying Lemma \ref{lemconslaplace} and \ref{lemmabdifference} results in the following bounds for the consistency errors.

\begin{lemma} \label{lemmaconsistencyerror}
Let $(\bu, p)  \in \bV_T \times L^2_0(\Gamma)$ be the unique solution of problem \eqref{contform} and assume that $(\bu, p)  \in H^2(\Gamma)^3 \times H^1(\Gamma)$. We further assume that the data errors satisfy $\Vert \vert \bB\vert \bbf^e - \bbf_h \Vert_{L^2(\Gamma_h)} \lesssim h^{k+1} \Vert \bbf \Vert_{L^2(\Gamma)}$ and $\Vert \vert \bB\vert g^e - g_h \Vert_{L^2(\Gamma_h)} \lesssim h^{k+1} \Vert g \Vert_{L^2(\Gamma)}$. The following holds:
\begin{multline} \label{consistencyp1}
\sup_{(\bv_h,q_h)  \in \bU_h \times Q_h} \frac{\vert \mathcal{A}_h((\bu^e,p^e), (\bv_h,q_h)) - (\bbf_h, \bv_h)_{L^2(\Gamma_h)} + (g_h, q_h)_{L^2(\Gamma_h)} \vert}{\left(\Vert \bv_h \Vert_{A}^2 + \Vert q_h \Vert_{M}^2\right)^\frac{1}{2}} \\
\lesssim h^{k}\left(\Vert \bu \Vert_{H^2(\Gamma)} +  \Vert p \Vert_{H^1(\Gamma)}\right) + h^{k+1} \left(\Vert \bbf \Vert_{L^2(\Gamma)} + \Vert g \Vert_{L^2(\Gamma)}\right).
\end{multline}
\end{lemma}

\subsection{Finite element error bound} \label{sectDiscrerror}

We combine the Strang-Lemma~\ref{stranglemma} and the bounds for the approximation error and the consistency error to obtain a bound for the discretization error in the energy norm.

\begin{theorem} \label{energyerrorboundph1}
Let $(\bu, p)  \in \bV_T \times L^2_0(\Gamma)$ be the unique solution of problem \eqref{contform} and assume that $(\bu, p)  \in H^2(\Gamma)^3 \times H^1(\Gamma)$. Let $(\bu_h, p_h) \in \bU_h \times Q_h$ be the unique solution of the discrete problem \eqref{discreteform1} with parameters as in  \eqref{choicerho}. We further assume that the data errors satisfy $\Vert \vert \bB\vert \bbf^e - \bbf_h \Vert_{L^2(\Gamma_h)} \lesssim h^{k+1} \Vert \bbf \Vert_{L^2(\Gamma)}$ and $\Vert \vert \bB\vert g^e - g_h \Vert_{L^2(\Gamma_h)} \lesssim h^{k+1} \Vert g \Vert_{L^2(\Gamma)}$. Then the following error bound holds:
	\begin{equation} \label{bound1} \begin{split}
\Vert \bu^e - \bu_{h} \Vert_{A} + \Vert p^e - p_{h} \Vert_{M} & \lesssim h^{k} \left(\Vert \bu \Vert_{H^{k+1}(\Gamma)} +  \Vert p \Vert_{H^{k}(\Gamma)} \right) \\
 &\qquad +h^{k+1} \left(\Vert \bbf \Vert_{L^2(\Gamma)} + \Vert g \Vert_{L^2(\Gamma)}\right).
	\end{split}
\end{equation}
\end{theorem}

\section{Numerical experiments}\label{sectNumer}
Results of numerical experiments (for different surfaces $\Gamma$)  that confirm the optimal order of convergence of the trace Taylor--Hood finite method  for  $k=2$ and $k=3$ are presented in \cite{jankuhn2019higher}. These results show optimal convergence behavior, not only in the energy norm but also in the $L^2$-norm. In that paper, one can also find numerical results for an \emph{inconsistent variant} of the method in which an approximation $\bH_h$ of the Weingarten mapping is \emph{not} needed. In \cite{olshanskii2019inf}   results of a numerical experiment with $k=2$ are presented which illustrate that without the pressure normal stabilization term, i.e., using $\rho_p=0$, the trace Taylor--Hood pair is \emph{not} inf-sup stable.
Below we  present results of two further numerical experiments. In Section \ref{sectinfsup} we
numerically confirm the inf-sup stability of the trace Taylor--Hood pair $\bU_h \times Q_h$ for $k=2,3,4,5$. The results show that the (best) inf-sup constant is (in this $k$ range) essentially independent of $k$.
In Section \ref{sectKelvin}
we apply our method to the Kelvin-Helmholtz instability problem, which
illustrates the potential of the method.

\subsection{Inf-sup constant} \label{sectinfsup}
We consider the Stokes problem on the unit sphere, characterized as the zero level of
the distance function $\phi(x) = \sqrt{x_1^2+x_2^2+x_3^2} - 1$, $x=
(x_1,x_2,x_3)^T$.  The discretizaton \eqref{discreteform1} is implemented in  NGSolve
\cite{ngsolve} with the surface embedded in a domain $\Omega=(-\tfrac{5}{3},\tfrac{5}{3})^3$, a coarsest mesh-size of $h_0 = 0.5$ and several uniform refinements (only of tetrahedra intersected by the surface).
We use parameter values $\rho_u=h^{-1}$, $\rho_p=h$, $\eta=h^{-2}$. The resulting
discrete saddle point problem  and its pressure Schur complement are of the form
\[
\mathbf{ \mathcal{A}}:=\left[\begin{matrix}
\bA & \bB^T \\
\bB & -\bC
\end{matrix}\right],\quad \bS=\bB \bA^{-1} \bB^T + \bC.
\]
Let $\bM$ be the symmetric positive definite matrix corresponding to the scalar product that induces the norm $\|\cdot\|_M$ used in the pressure space $Q_h$, cf. \eqref{defnorms}. We consider the generalized eigenvalue problem
\begin{equation*}
\begin{split}
\bS \vec p = \lambda \bM \vec p.
\end{split}
\end{equation*}
The smallest strictly positive  eigenvalue, denoted by $\lambda = \lambda_{\min}$ is related to the best possible inf-sup constant in \eqref{infsup}  through $\frac12 c_0^2 \leq \lambda_{\min} \leq 2 c_0^2$. For the
computation of the eigenvalues we use SciPy \cite{SciPy}.  Further details concerning this eigenvalue computation are given in \cite{olshanskii2019inf}.
 In Table~\ref{table2} we show computed $\lambda_{\min}$ values for several grid refinements and polynomial degree $k=2,\ldots,5$. (Due to computational limitations the last entries in the fifth and sixth column are not included).
\begin{table}[ht!]
  \centering
 \begin{tabular}{|c||c|c|c|c|c|}
\hline
 \textit{l} & $\vect P_1$--$P_{1}$ & $\vect P_2$--$P_{1}$ & $\vect P_3$--$P_{2}$ & $\vect P_4$--$P_{3}$ & $\vect P_5$--$P_{4}$ \\ \hline
1 & 0.84227 & 0.98940 & 0.98996 & 0.98999 & 0.99045 \\
2 & 0.73001 & 0.98312 & 0.98247 & 0.98346 & 0.98534 \\
3 & 0.65410 & 0.97322 & 0.97363 & 0.97583 & 0.97798 \\
4 & 0.52795 & 0.96089 & 0.96563 & 0.96595 & 0.96923 \\
5 & 0.39170 & 0.94002 & 0.93990 & 0.93486 & -- \\
6 & 0.27037 & 0.94585 & 0.94670 & -- & -- \\ \hline
 \end{tabular}
\caption{Smallest strictly positive eigenvalue $\lambda_{\min}$.} \label{table2}
\end{table}

As predicted by the theoretical analysis, the eigenvalue $\lambda_{\min}$ remains bounded away from zero as the grid is refined. We also observe that $\lambda_{\min}$ remains essentially constant if one increases $k$. This robustness property does not follow from our analysis. In the second column of the table we show the result for the $\vect P_1$--$P_{1}$ pair of trace finite element spaces (with $P_1$ approximation of the surface). The results  indicate  that, as expected,  this
pair is not inf-sup stable if we use (only) the normal derivative pressure stabilization $\tilde s_h(\cdot,\cdot)$. If one uses an additional Brezzi-Pitk\"aranta type stabilization this pair becomes inf-sup stable, as is shown in \cite{olshanskii2018finite}.

\subsection{Kelvin--Helmholtz instability on a sphere} \label{sectKelvin}
To demonstrate the performance of the method under more general circumstances not covered by the presented analysis, we further consider a classical problem of the Kelvin--Helmholtz instability  in a mixing layer of isothermal incompressible viscous flow at high Reynolds number.
For a detailed discussion of the problem in a 2D periodic square, which can be seen as a planar analogue of our setup, we refer to~\cite{schroeder2019reference} and the references therein. There are almost no numerical studies of Kelvin--Helmholtz instability for surface fluids; examples of a cylinder and a sphere are treated in ~\cite{lederer2019divergence}, where a higher order $H$(div)-conforming finite element method is applied on triangulated surfaces. We follow that paper to design  our numerical experiment.

For~$\Gamma=S^2$ , let $\xi$ and $\zeta$ to be {renormalized} azimuthal and polar coordinates, respectively: $-1/2 \le \xi, \zeta < 1/2$. The corresponding directions are~$\vect e_\xi \coloneqq \nabla_{\Gamma}\xi/\|\nabla_{\Gamma}\xi\|$ and $\vect e_\zeta \coloneqq \nabla_{\Gamma}\zeta/\|\nabla_{\Gamma}\zeta\|$. Consider the initial velocity field
\begin{align}\begin{split} \label{kh:ini}
	{\vect u}_0(\xi, \zeta) &\coloneqq \tanh(2\,\zeta/\delta_0)\,r(\zeta)\,\vect e_\xi + c_n\vCurl_\Gamma\psi, \\
	\psi(\xi, \zeta) &\coloneqq e^{-(\zeta/\delta_0)^2}\,\big(a_a\cos(m_a\,\pi\,\xi) + a_b\cos(m_b\,\pi\,\zeta)\big),
\end{split}\end{align}
where $r$ is the distance from $\Gamma$ to the $z$-axis. We take $\delta_0 \coloneqq 0.05$ (for $|z| \gtrsim \delta_0$ the velocity field is close to a rigid body rotation around the $z$-axis), $c_n \coloneqq 10^{-2}$ (perturbation parameter), and $a_a = 1$, $m_a = 16$, $a_b = 0.1$, $m_b = 20$ (perturbation magnitudes and frequencies). Note that ${\vect u}_0$ is tangential by construction, ${\vect u}_0\cdot\bn=0$. The initial velocity field is illustrated in Figure~\ref{fig:kh:ini}.

\begin{figure}[ht!]
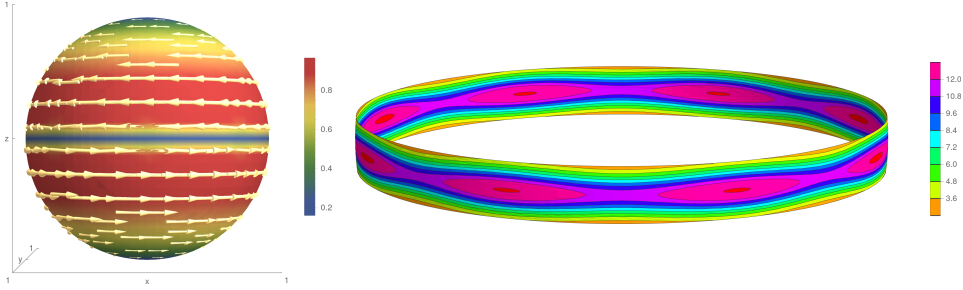

	\par\bigskip
	\centering
	\begin{subfigure}{.35\linewidth}
		\centering
		\includegraphicsw[.95]{{kh_ini_u_8.cropped}.png}
	\end{subfigure}%
	\begin{subfigure}{.65\linewidth}
		\centering
		\includegraphicsw[.95]{{kh_ini_w_8.cropped}.png}
	\end{subfigure}%
	\caption{Left: Initial velocity field~$\vect u_0$ from \eqref{kh:ini}. Right: The initial vorticity, $\Curl_\Gamma \vect u_0$, in the strip $|z| < 2\,\delta_0$. We see that the initial perturbation consists of 8 vortices squeezed around equator.}
	\label{fig:kh:ini}		
\end{figure}

Compared to the surface Stokes problem~\eqref{eqstrong}, the surface Navier--Stokes equations considered in this experiment are time-dependent and include inertia terms:
\begin{align}\begin{split}\label{split:tang}
	\vect P \frac{\mbox d\vect u}{\mbox d t} - 2\nu\,\vect P\divG(E(\vect u)) + \nabla_\Gamma p &= \vect 0,\\
	\divG \vect u &= 0,
\end{split}\end{align}
where~$\frac{ {\rm d} \vect u}{{\rm d} t} = \frac{\partial {\vect u}}{\partial t} + (\vect u\cdot\nabla)\vect u$ is the material derivative.  For the unit sphere and initial condition such that $\|\vect u\|_{\LInfSpace} \simeq 1$,  we have a Reynolds number~$\text{Re} \simeq \nu^{-1}\delta_0$. In our numerical tests we set $\nu = \frac12 10^{-5}$, resulting in $\text{Re}=10^4$.

We note that equations~\eqref{split:tang} follows by tangential projection of a fluid system governing the evolution of a viscous material layer under the assumption of vanishing radial motions; see~\cite{Jankuhn1}.  The operator $\vect P \frac{{\rm d}\vect u}{{\rm d} t}$ can be seen as covariant material derivative. One checks the identity  $\vect P \frac{{\rm d}\vect u}{{\rm d} t} = \frac{\partial\vect u}{\partial t}+(\nabla_\Gamma \bu)\bu$ for a tangential vector field $\bu$, which we further use in the finite element formulation.

We outline the discretization approach used for the simulation of this surface Navier--Stokes problem. The  trace $\vect P_2$--$P_1$ Taylor--Hood finite element method as described in this paper, cf. \eqref{discreteform1},  is applied for the spatial discretization. Discretization parameters were chosen as ~$\rho_p = h$, $\rho_u = h^{-1}$, and~$\eta = h^{-2}$, cf.  \eqref{choicerho}. We use the BDF2 scheme to approximate  $\frac{\partial\vect u}{\partial t}$
and linearize the inertia term at  $t^n$ as $(\nabla_\Gamma \bu(t^n))\bu(t^n)\approx(\nabla_\Gamma \bu(t^n))\bw$, where $\bw$ is the linear extrapolation of velocity fields from two previous time nodes, $t^{n-1}$ and $t^{n-2}$. The grad-div stabilization term~\cite{olshanskii2004grad}, $\gamma\int_{\Gamma^h}\Tr E(\vect u)\Tr E(\vect v)\,ds$ with $\gamma=1$, is added to the finite element formulation to better enforce divergence free condition. This stabilization also facilitates the construction of preconditioners for the resulting algebraic systems~\cite{heister2013efficient}.
No further stabilizing terms, e.g., of streamline diffusion type, were included in the method, since the computed solution does not reveal any spurious modes.

The method is implemented in the DROPS package \cite{DROPS}.
For this series of experiments,  an initial triangulation $\T_{h_0}$  was build by  dividing $\Omega=(-\frac53,\frac53)^3$ into $2^3$ cubes and further splitting each cube into 6 tetrahedra with  $h_0 = \frac53$. Further, the mesh is  refined only close to the surface, and $\ell \in \mathbb{N}$ denotes the level of refinement so that $h_\ell = \frac53\,2^{-\ell}$.
%$\Gamma$ is chosen to be the unit sphere, $\phi(\vect x) \coloneqq  \|\vect x\|^2 - 1$, and with this choice of~$\phi \in P_2$ the surface is represented exactly, $\Gamma_h = \Gamma$; Our main concern for the example is the energy dissipation, and we would like to eliminate all the concerns for geometric errors here.

\begin{figure}[ht!]
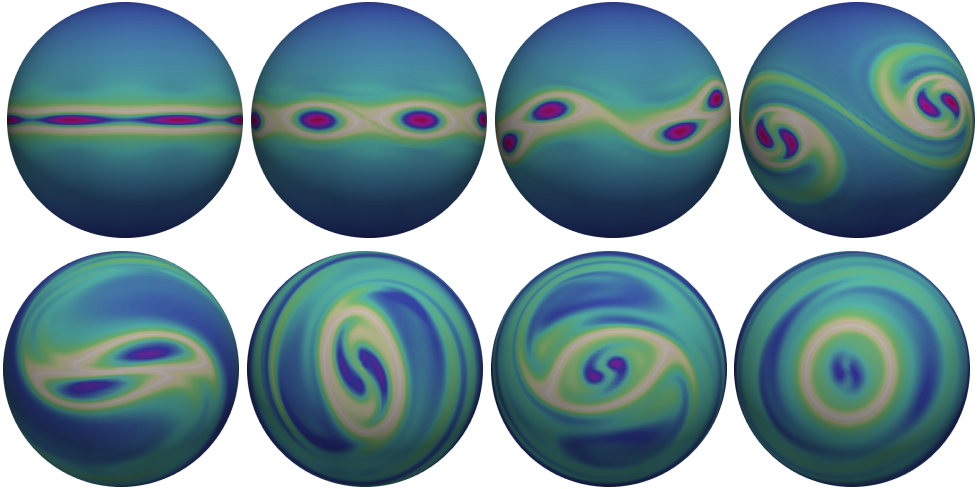

	\centering
\href{https://youtu.be/EwF3vCgFhuI}{
	%\begin{subfigure}{.245\linewidth}
		\includegraphicsw[.24]{{kh_t=0.cropped}.png}
	%\end{subfigure}%
	%\begin{subfigure}{.245\linewidth}
		\includegraphicsw[.24]{{kh_t=2.5.cropped}.png}
	%\end{subfigure}%
	%\begin{subfigure}{.245\linewidth}
		\includegraphicsw[.24]{{kh_t=5.cropped}.png}
	%\end{subfigure}%
	%\begin{subfigure}{.245\linewidth}
		\includegraphicsw[.24]{{kh_t=6.25.cropped}.png}
	%\end{subfigure}%
%\newline
\vskip1ex
%\par\bigskip
	%\begin{subfigure}{.245\linewidth}
		\includegraphicsw[.24]{{kh_t=10.cropped}.png}
	%\end{subfigure}%
	%\begin{subfigure}{.245\linewidth}
		\includegraphicsw[.24]{{kh_t=12.5.cropped}.png}
	%\end{subfigure}%
	%\begin{subfigure}{.245\linewidth}
		\includegraphicsw[.24]{{kh_t=15.cropped}.png}
	%\end{subfigure}%
	%\begin{subfigure}{.245\linewidth}
		\includegraphicsw[.24]{{kh_t=20.cropped}.png}
	%\end{subfigure}%
              }
	\caption{Snapshots of surface vorticity~${w_h = \Curl_{\Gamma_h}\vect u_h}$ for~${t \in \{ 0, 2.5, 5, 6.25, 10, 12.5, 15, 20 \}}$, ${h = 2.6\times 10^{-2}}$. Click any picture for a full animation.}
	\label{fig:kh:curl}		
\end{figure}

We perform numerical simulations for mesh levels $\ell = 4, 5, 6$. The DROPS  package currently does not support parametric elements, so for a sufficiently accurate numerical integration we use a piecewise linear approximation of $\Gamma_{h}$  with $m_\ell$ levels of local refinement, where $m_4=2$, $m_5=4$, $m_6=8$; see section~6.3 in \cite{olshanskii2019inf} for further details. The time interval is fixed to be~$\lbrack0,  20\rbrack$. We use uniform time stepping with $\Delta t=1/16$, $1/32$ and $1/64$ for mesh levels $4,5$ and $6$, respectively.

Figure~\ref{fig:kh:curl} shows several snapshots of the surface vorticity, $w_h = \mbox{curl}_\Gamma\bu_h$, computed on the finest mesh level 6.  The trace $\mathbf{P}_2$--$P_1$  finite element method that we use reproduces qualitatively correct flow dynamics that follows the well known pattern of the planar Kelvin--Helmholtz instability development: we see the initial vortices formation in the layer followed by pairing and self-organization into two large counter-rotating vortices. Conservation of the initial zero angular momentum prevents further pairing. The two remaining vortices  should decay for $t\to+\infty$ due to energy dissipation.

We next assess the method by monitoring the energy dissipation of the computed solutions on three subsequent levels. To have a better insight into the expected behaviour, we note that the initial velocity $\bu_0$ is $L^2$-orthogonal to all rigid tangential motions of $\Gamma$, functions from $E = \{\bv\in\bV_T\,:\,E(\bv)=0\}$. It is straightforward to check that a velocity field $\bu$ that solving \eqref{split:tang}  preserves this orthogonality condition for all $t>0$ and hence it satisfies the following Korn inequality:
\begin{equation}\label{Korn2}
\|\vect u\|_{\LTwoSpace} \le C_K(\Gamma)\,\|E(\vect u)\|_{\LTwoSpace}.
\end{equation}
For the total kinetic energy  $\KinEn(t) = \frac 12 \|\vect u(\cdot, t)\|^2_{\LTwoSpace}$, testing    \eqref{split:tang} with $\vect v = \vect u$ and applying \eqref{Korn2} leads to the following identity and a corresponding energy bound:
\begin{equation*}
	\frac{\diff{\KinEn(t)}}{\diff{t}} = -2\nu\,\|E(\vect u(t))\|^2_{\LTwoSpace} \le -\frac{4\,\nu}{C_K^2(\Gamma)}\,\KinEn(t)\quad\Longrightarrow~\KinEn(t)\le \KinEn(0)\exp\left(-\frac{4\,\nu\,t}{C_K^2(\Gamma)}\right).
\end{equation*}
We outline an approach for estimating the Korn constant $C_K(\Gamma)$. The best  value of this constant is obtained if  $C_K(\Gamma)^{-2}$ is the smallest strictly positive eigenvalue of the diffusion operator $-\bP \divG (E(\cdot))$ restricted to the space of tangential divergence free vector fields, cf. \eqref{Korn2}.
We have the following relation between this surface diffusion operator and the Hodge-de Rham operator $\Delta^H_\Gamma$ (see, eq.~(3.18) in \cite{Jankuhn1}):
\begin{equation}\label{aux1328}
-2\bP \divG (E(\bv))=\Delta^H_\Gamma\bv- 2K\bv,\quad\text{for}~\bv\in\bV_T,~\text{s.t.}~\divG\bv=0,
\end{equation}
where $K$ is the Gauss curvature ($K=1$ for $\Gamma=S^2$). The eigenvalues of $\Delta^H_\Gamma$ for the unit sphere are given by $\lambda_k(\Delta^H_\Gamma)=k(k+1)$, $k=1,2,\dots$,~\cite[p.349]{chow2007ricci}. The tangential rigid motions are eigenfunctions corresponding to $\lambda_1$.  Hence, we estimate:
\[
\begin{split}
C_K(\Gamma)^{-2}&=\inf_{\bv\in\bV_T/E\atop \divG\bv=0 }\frac{\|E(\bv)\|_{L^2(\Gamma)}^2}{\|\bv\|_{L^2(\Gamma)}^2}
=\inf_{\bv\in\bV_T/E\atop \divG\bv=0 }\frac{\frac12\langle\Delta^H_\Gamma\bv- 2K\bv,\bv\rangle}{\|\bv\|_{L^2(\Gamma)}^2} \\ &\ge
\inf_{\bv\in\bV_T/E}\frac{\frac12\langle\Delta^H_\Gamma\bv- 2K\bv,\bv\rangle}{\|\bv\|_{L^2(\Gamma)}^2} = \frac12(\lambda_2(\Delta^H_\Gamma)-2)=2,
\end{split}
\]
resulting in   $C_K(\Gamma)^2\le\frac12$\footnote{Results of numerical experiments (not included), strongly suggest that $C_K(\Gamma)^2=\frac12$ for $\Gamma=S^2$.}.
%the smallest strictly positive eigenvalue of $-\bP \divG (E(\cdot))$ is $  \frac12(\lambda_2(\Delta^H_\Gamma)-2)=2$, resulting in   $C_K(\Gamma)^2=\frac12$.
Substituting this in the above estimate for the kinetic energy, we arrive at the bound
\begin{equation}\label{kin_decay}
	\KinEn(t)\le \KinEn(0)\exp\left(-8\nu\,t\right)=\KinEn(0)\exp\left(-4 \cdot 10^{-5}\,t\right) .
\end{equation}

\begin{figure}[h]
	\centering
	\begin{subfigure}{.5\linewidth}
		\centering
		\includegraphicsw{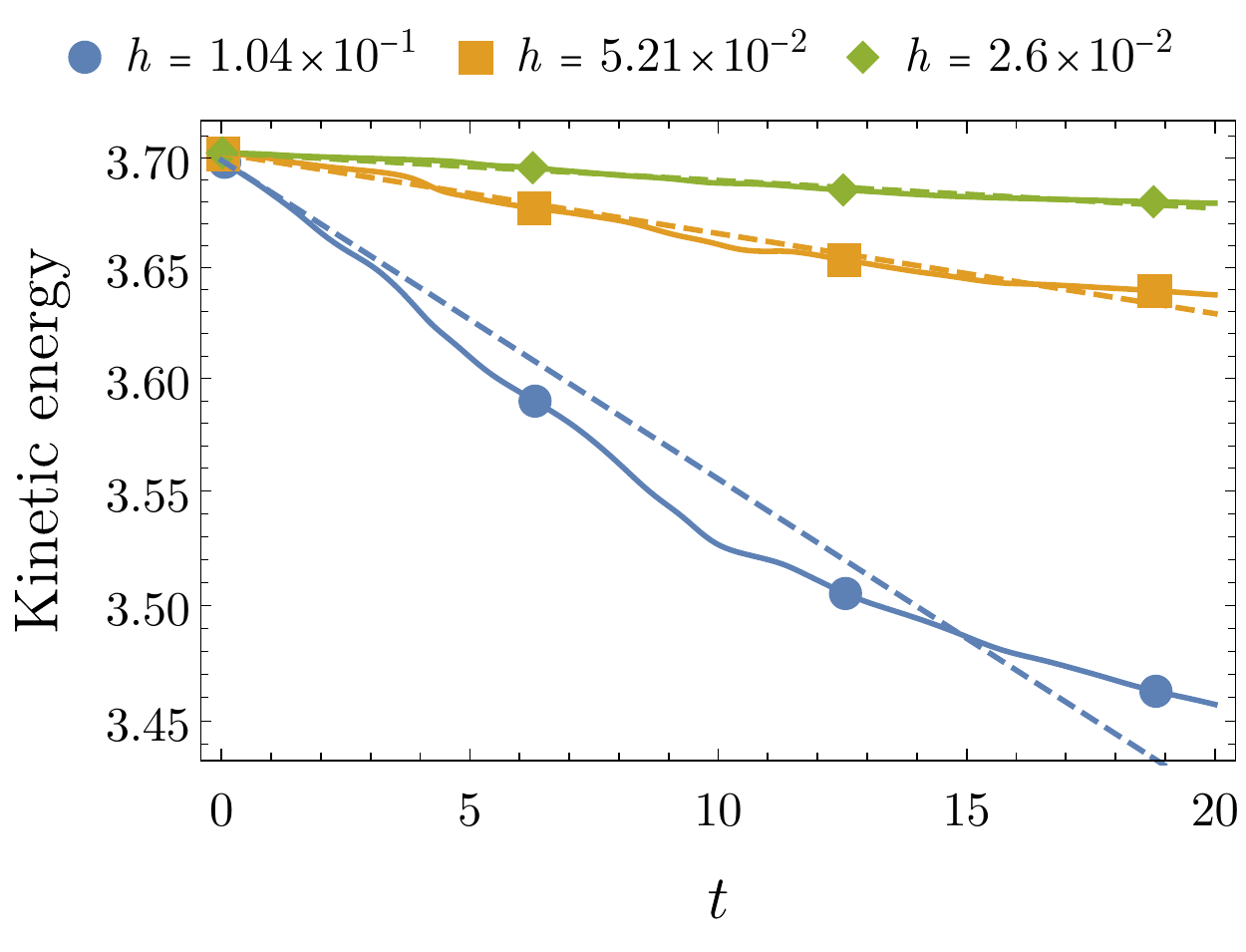}
	\end{subfigure}%
	\begin{subfigure}{.5\linewidth}
		\centering\small
		\begin{tabular}[1.2]{|c|c|}
			\hline
			$h$ & $\alpha$  \\
			\hline
			$1.04\times 10^{-1}$ & $3.96\times 10^{-3}$ \\%& \\
			\hline
			$5.21\times 10^{-2}$ & $9.96\times 10^{-4}$ \\%& $1.99$ \\
			\hline
			$2.6\times 10^{-2}$  & $3.44\times 10^{-4}$ \\% & $1.53$ \\
			\hline
		\end{tabular}
	\end{subfigure}%
	\caption{Left: Numerical kinetic energies~$\mathcal{E}_h(t)=\frac 12 \|\vect u_h(\cdot, t)\|^2_{L^2(\Gamma_h)}$ as functions of time for~$\ell=4,5,6$ (straight lines) and corresponding exponential fitting  (dashed lines). Right: Values of the exponent $\alpha$ in the fitting function $C\exp(-\alpha t)$.}
	\label{fig:kh:kinetic}		
\end{figure}

In Figure~\ref{fig:kh:kinetic} we show the kinetic energy plots for the computed solutions together with exponential fitting. There are two obvious reasons for the computed energy to decay faster than the upper estimate \eqref{kin_decay} suggests: the presence of numerical diffusion and the persistence of higher harmonics in the true solution. On the finest mesh the numerical solution looses about  $0.5\%$ of kinetic energy up to the point when the solution is dominated by two counter-rotating vortices.
This compares well to results computed with a higher order method in \cite{schroeder2019reference} for the planar case with $Re=10^4$.

\subsection*{Acknowledgment} The authors Th. Jankuhn and A. Reusken wish to thank the German Research Foundation (DFG) for financial support within the Research Unit ``Vector- and tensor valued surface PDEs'' (FOR 3013) with project no. RE 1461/11-1. M.O. and A.Zh. were partially supported by NSF through the Division of Mathematical Sciences grant 1717516.

\bibliographystyle{siam}
\bibliography{main}{}

\end{document}